\documentclass[11pt,a4paper]{article}
\usepackage{titling,lipsum}
\usepackage{caption}
\usepackage{subcaption}
\usepackage{graphicx} 
\usepackage{url}      
\usepackage{amssymb}
\usepackage{amsmath}  
\usepackage{amsthm}
\usepackage{amsfonts}%
\usepackage{float}
\newtheorem{theorem}{Theorem}

\newtheorem{lemma}{Lemma}
\theoremstyle{remark}
\newtheorem{remark}{Remark}

\begin{document}
\date{}
\title{Block monotone iterative methods for solving coupled systems of nonlinear elliptic problems}
\author{Mohamed Al-Sultani
\\
Institute of Fundamental Sciences, Massey University,\\
 Palmerston North, New Zealand\\
E-mail: M.Al-Sultani@massey.ac.nz}\maketitle
\begin{abstract}
 This paper investigates numerical methods for solving  coupled system of nonlinear elliptic problems. We utilize block monotone iterative methods based on   Jacobi and Gauss--Seidel methods to solve difference schemes which approximate the coupled system of nonlinear elliptic problems, where reaction functions are quasimonotone nondecreasing. In the view of upper and lower solutions method, two monotone upper and lower sequences of solutions are constructed, where the monotone property ensures  the theorem on existence of  solutions to problems with quasimonotone  nondecreasing reaction functions. Construction of initial upper and lower solutions is presented. The sequences of solutions generated by the block Gauss--Seidel  method converge not slower than by the block Jacobi method.
 \end{abstract}
\section{Introduction}
Several problems in the chemical, physical and engineering sciences are characterized by coupled systems of nonlinear elliptic equations \cite{P92}. 
 In this paper, we construct  block monotone iterative methods for solving  the coupled system of nonlinear elliptic equations
\begin{equation}\label{cont1a}
  -\mbox{L}_{\alpha}u_{\alpha}(x,y)+f_{\alpha}(x,y,u)=0,\quad (x,y)\in \omega,\quad \alpha=1,2,
\end{equation}
$$
\omega=\{(x,y):0<x<1,\quad 0<y<1\},
$$
\begin{equation*}
  u(x,y)=g(x,y),\quad (x,y)\in \partial \omega,
\end{equation*}
where $u=(u_{1},u_{2})$, $f=(f_{1},f_{2})$, $g=(g_{1},g_{2})$, and $\partial \omega$ is the boundary of $\omega$. The differential operators $\mbox{L}_{\alpha}$, $\alpha=1,2$, are defined by
$$
\mbox{L}_{\alpha}u_{\alpha}(x,y)\equiv \varepsilon_{\alpha}(u_{\alpha,xx}+u_{\alpha,yy})-v_{\alpha}(x,y)(u_{\alpha,x}+u_{\alpha,y}),
$$
where
$\varepsilon_{\alpha}$, $\alpha=1,2$, are positive constants. It is assumed that the functions $f_{\alpha}$, $g_{\alpha}$, $v_{\alpha}$, $\alpha=1,2$,
are smooth in their respective domains.

For treating  such nonlinear problems numerically  using finite difference or finite element methods, the nonlinear problems are approximated by difference schemes which lead to nonlinear systems of algebraic equations.
The main  mathematical concern is to investigate whether these systems have  a solution and  to find  efficient, stable and computationally effective methods for solving these discrete systems. The method of upper and lower solutions  and its associated monotone iterations is a fruitful method for solving nonlinear difference schemes. The method of upper and lower solutions and its associated monotone iterations in the  continuous case was originally developed  for solving nonlinear elliptic and parabolic problems (see (\cite{P92}, \cite{s72}, for details). By using upper and lower  solutions as two initial iterations, one can construct two monotone sequences which converge monotonically from above and below to a solution of the problem. This monotone property ensures the theorem on existence and uniqueness  of solutions to problems.
\par
In the case of nonlinear difference schemes, the method of upper and lower solutions has been developed and applied for solving elliptic and parabolic problems (see \cite{I06}, \cite{pa65}, \cite{p93}, for details).
\par
Block monotone iterative methods, based on the method of upper and lower solutions, have been used for solving nonlinear scalar elliptic equations \cite{II05}, \cite{Pa82}, \cite{p95},  \cite{p&X03}, \cite{w11}. The basic idea of the block monotone iterative  methods is to decompose of a two dimensional problem into a series of one dimensional two-point boundary value problems. Each of one dimensional problem  can be solved efficiently by a standard computational scheme such as the Thomas algorithm.

The application of the method of upper and lower solutions  to coupled systems has more complexity. For solving coupled systems of nonlinear elliptic equations, monotone iterative methods based on  the method of upper and lower solutions have been developed  in \cite{p90} for continuous problems  and in \cite{I07}, \cite{I08}, \cite{I10}, \cite{w09} for discrete problems.

The aim of this article is to construct and  investigate  block monotone iterative methods  based on Jacobi and Gauss--Seidel methods  for solving coupled systems of nonlinear elliptic equations  with quasimonotone nondecreasing reaction functions $f_{\alpha}$, $\alpha=1,2$, which satisfy the inequalities
\begin{equation*}
-\frac{\partial f_{\alpha}}{\partial u_{\alpha^{\prime}}}\geq0,\quad (x,y)\in \overline{\omega},\quad \alpha^{\prime}\neq \alpha,\quad \alpha=1,2.
\end{equation*}
 The article is structured as follows. Section 2 deals with some properties of solutions to system (\ref{cont1a}). In section 3, we consider a nonlinear difference scheme which approximates the nonlinear elliptic problem (\ref{cont1a}). The construction of the block monotone  Jacobi and  Gauss--Seidel iterative methods is located in Section 4. Section 5 exhibits the construction of initial upper and lower solutions which are used as initial iterations in the monotone iterative methods. Finally, in Section 6, the convergence rate of the block monotone Jacobi and  Gauss--Seidel  iterative methods are compared.
\section{Properties of solutions to system (\ref{cont1a})}
Two vector functions $\widetilde{u}(x,y)=(\widetilde{u}_{1},\widetilde{u}_{2})$ and $\widehat{u}(x,y)=(\widehat{u}_{1},\widehat{u}_{2})$, are called, respectively, upper and lower solutions to (\ref{cont1a}), if they satisfy the inequalities
\begin{equation}\label{eq:cont1}
\widetilde{u}(x,y)\geq\widehat{u}(x,y),\quad (x,y)\in \overline{\omega},
\end{equation}
\begin{equation*}
-\mbox{L}_{\alpha}\widetilde{u}_{\alpha}(x,y)+f_{\alpha}(x,y,\widetilde{u})\geq0,\quad (x,y)\in \omega,
\end{equation*}
\begin{equation*}
 -\mbox{L}_{\alpha}\widehat{u}_{\alpha}(x,y)+f_{\alpha}(x,y,\widehat{u})\leq0,\quad  (x,y)\in \omega,
\end{equation*}
\begin{equation*}
\widehat{u}(x,y)\leq g(x,y) \leq \widetilde{u}(x,y),\quad (x,y)\in \partial \omega.
\end{equation*}
For a given ordered upper $\widetilde{u}$ and lower $\widehat{u}$ solutions, a sector $\langle \widehat{u}, \widetilde{u}\rangle$ is defined as follows
\begin{equation*}
  \langle \widehat{u}, \widetilde{u}\rangle=\left\{u(x,y);\quad \widehat{u}(x,y)\leq u(x,y)\leq \widetilde{u}(x,y),\quad (x,y)\in \overline{\omega}\right\}.
\end{equation*}
On the sector $\langle \widehat{u}, \widetilde{u}\rangle$, the vector  function $f(x,y,u)$ is assumed to satisfy the constraints
\begin{equation}\label{cont2}
  \frac{\partial f_{\alpha}(x,y,u)}{\partial u_{\alpha}}\leq c_{\alpha}(x,y),\quad u\in \langle \widehat{u},\widetilde{u}\rangle,\quad (x,y)\in \overline{\omega},\quad \alpha=1,2,
\end{equation}
\begin{equation}\label{eq:cont3}
-\frac{\partial f_{\alpha}(x,y,u)}{\partial u_{\alpha^{\prime}}}\geq0,\quad u\in \langle \widehat{u},\widetilde{u}\rangle, \quad (x,y)\in \overline{\omega},\quad \alpha^{\prime}\neq \alpha,\quad \alpha=1,2,
\end{equation}
where $c_{\alpha}(x,y)$, $\alpha=1,2$, are non-negative bounded functions.
The vector function $f(x,y,u)$ is called quasimonotone nondecreasing on $\langle \widehat{u},\widetilde{u}\rangle$, if it satisfies (\ref{eq:cont3}).

Consider the following iterative method for solving the nonlinear system (\ref{cont1a})
\begin{equation}\label{ccont1b}
 -\mbox{L}_{\alpha}u_{\alpha}^{(n)}(x,y)+c_{\alpha}(x,y)u^{(n)}_{\alpha}(x,y)=-f_{\alpha}(x,y,u^{(n-1)}),\quad (x,y)\in \omega,\ \alpha=1,2,
\end{equation}
\begin{equation*}
  u^{(n)}(x,y)=g(x,y),\quad (x,y)\in \partial \omega,
\end{equation*}
where $c_{\alpha}(x,y)$, $\alpha=1,2$, are defined in (\ref{cont2}).
\begin{theorem}\label{contth1}
Assume that the vector function $f(x,y,u)$ in (\ref{cont1a}) satisfies (\ref{cont2}) and (\ref{eq:cont3}). Let $\widetilde{u}=(\widetilde{u}_{1},\widetilde{u}_{2})$ and $\widehat{u}=(\widehat{u}_{1},\widehat{u}_{2})$ be ordered upper and lower solutions. Then the upper $\{\overline{u}^{(n)}\}$ and lower $\{\underline{u}^{(n)}\}$ sequences, generated by (\ref{ccont1b}) with
$\overline{u}^{(0)}=(\widetilde{u}_{1},\widetilde{u}_{2})$ and $\underline{u}^{(0)}=(\widehat{u}_{1},\widehat{u}_{2})$ converge monotonically, respectively, from above to a maximal solution $\overline{u}$ and from below to a minimal solution $\underline{u}$, such that
\begin{equation}\label{cont4}
\widehat{u}\leq \underline{u}^{(n-1)}\leq \underline{u}^{(n)}\leq \underline{u}\leq \overline{u}\leq\overline{u}^{(n)}\leq \overline{u}^{(n-1)}\leq \widetilde{u}
\quad \mbox{in} \quad \overline{\omega},\quad n\geq1.
\end{equation}
    If $s=(s_{1},s_{2})$ is  any other solution in $\langle \widehat{u},\widetilde{u}\rangle$, then $\underline{u}\leq s\leq \overline{u}$.
\end{theorem}
The proof of the theorem can be found in \cite{P92}.
\section{The nonlinear difference scheme}
On  $\overline{\omega}$, we introduce a rectangular mesh $\overline{\Omega}^{h}=\overline{\Omega}^{hx}\times\overline{\Omega}^{hy}$:
\begin{equation*}
\overline{\Omega}^{hx}=\{x_{i},\quad  i=0,1,\ldots, N_{x};\quad x_{0}=0, \quad x_{N_{x}}=1; \quad  h_{x}=x_{i+1}-x_{i}\},
\end{equation*}
\begin{equation*}
\overline{\Omega}^{hy}=\{y_{j},\quad  j=0,1,\ldots, N_{y};\quad y_{0}=0, \quad y_{N_{y}}=1; \quad h_{y}=y_{j+1}-y_{j}\}.
\end{equation*}
For a mesh function $U(x_{i},y_{j})=(U_{1}(x_{i},y_{j}),U_{2}(x_{i},y_{j}))$, $(x_{i},y_{j}) \in \overline{\Omega}^{h}$, we use the  difference scheme
\begin{equation}\label{eq:2}
  \mathcal{L}_{\alpha,ij}U_{\alpha}(x_{i},y_{j})+f_{\alpha}(x_{i},y_{j},U)=0,\quad (x_{i},y_{j}) \in \Omega^{h},\quad \alpha=1,2,
  \end{equation}
  \begin{equation*}
 U(x_{i},y_{j})=g(x_{i},y_{j}),\quad (x_{i},y_{j})\in \partial \Omega^{h},
\end{equation*}
where $\partial \Omega^{h}$ is the boundary of the mesh  $\Omega^{h}$, and $\mathcal{L}_{\alpha,ij}U_{\alpha}(x_{i},y_{j})$, $\alpha=1,2$,  are defined by
\begin{eqnarray*}
 \mathcal{L}_{\alpha, ij}U_{\alpha}(x_{i},y_{j})&=&-\varepsilon_{\alpha}\left( D^{2}_{x}U_{\alpha}(x_{i},y_{j})+D^{2}_{y}U_{\alpha}(x_{i},y_{j})\right) \\
  &&+v_{\alpha}(x_{i},y_{j})\left(D^{1}_{x}U_{\alpha}(x_{i},y_{j})+D^{1}_{y}U_{\alpha}(x_{i},y_{j})\right).
\end{eqnarray*}
It is assumed that  $v_{\alpha}(x,y)\geq0$, $(x,y)\in\overline{\omega}$, $\alpha=1,2$.
  $D^{2}_{x}U_{\alpha}(x_{i},y_{j})$, $D^{2}_{y}U_{\alpha}(x_{i},y_{j})$ and $D^{1}_{x}U_{\alpha}(x_{i},y_{j})$, $D^{1}_{y}U_{\alpha}(x_{i},y_{j})$, $\alpha=1,2$, are, respectively, the central difference and backward difference approximations to the second and first derivatives:
\begin{equation*}
D^{2}_{x}U_{\alpha}(x_{i},y_{j}) =\frac{U_{\alpha,i-1,j}-2U_{\alpha,ij}+U_{\alpha,i+1,j}}{h_{x}^{2}},
\end{equation*}
$$
D^{2}_{y}U_{\alpha}(x_{i},y_{j}) =\frac{U_{\alpha,i,j-1}-2U_{\alpha,ij}+U_{\alpha,i,j+1}}{h_{y}^{2}},
$$
\begin{equation*}
D^{1}_{x}U_{\alpha}(x_{i},y_{j}) =\frac{U_{\alpha,ij}-U_{\alpha,i-1,j}}{h_{x}},\quad  D^{1}_{y}U_{\alpha}(x_{i},y_{j}) =\frac{U_{\alpha,ij}-U_{\alpha,i,j-1}}{h_{y}},
\end{equation*}
where $U_{\alpha,ij}\equiv U_{\alpha}(x_{i},y_{j})$.
\begin{remark}\label{remark1}
An approximation of the first derivatives $u_{x}$ and $u_{y}$ depends on the signs of $v_{\alpha}(x,y)$, $\alpha=1,2$.
When $v_{\alpha}(x,y)\leq0$, $\alpha =1,2$, then $u_{x}$ and $u_{y}$ are approximated by  forward difference formula. The first derivatives  $u_{x}$ and $u_{y}$ are approximated by  using both forward or backward difference formulae when $v_{\alpha}(x,y)$, $\alpha=1,2$, have variable signs.
\end{remark}
The vector mesh functions
\begin{equation*}
\widetilde{U}(x_{i},y_{j})=(\widetilde{U}_{1}(x_{i},y_{j}),\widetilde{U}_{2}(x_{i},y_{j})),\quad \widehat{U}(x_{i},y_{j})=(\widehat{U}_{1}(x_{i},y_{j}),\widehat{U}_{2}(x_{i},y_{j})),
\end{equation*}
$$
 (x_{i},y_{j})\in \overline{\Omega}^{h},
$$
are called ordered upper and lower solutions of (\ref{eq:2}), if they satisfy the inequalities
\begin{equation}\label{eq:3}
\widetilde{U}(x_{i},y_{j})\geq\widehat{U}(x_{i},y_{j}),\quad (x_{i},y_{j})\in \overline{\Omega}^{h},
\end{equation}
\begin{equation*}
\mathcal{L}_{\alpha,ij}\widetilde{U}_{\alpha}(x_{i},y_{j})+f_{\alpha}(x_{i},y_{j},\widetilde{U})\geq0,\quad (x_{i},y_{j})\in \Omega^{h},
\end{equation*}
\begin{equation*}
\mathcal{L}_{\alpha,ij}\widehat{U}_{\alpha}(x_{i},y_{j})+f_{\alpha}(x_{i},y_{j},\widehat{U})\leq0,\quad (x_{i},y_{j})\in \Omega^{h},
\end{equation*}
\begin{equation*}
  \widehat{U}(x_{i},y_{j})\leq g(x_{i},y_{j})\leq \widetilde{U}(x_{i},y_{j}),\quad (x_{i},y_{j})\in \partial \Omega^{h}.
\end{equation*}
For a given pair of ordered upper and lower solutions $\widetilde{U}(x_{i},y_{j})$ and $\widehat{U}(x_{i},y_{j})$, we define the sector
\begin{equation*}
  \langle \widehat{U},\widetilde{U}\rangle=\left\{U(x_{i},y_{j}):\widehat{U}(x_{i},y_{j})\leq U(x_{i},y_{j})\leq \widetilde{U}(x_{i},y_{j}),\quad (x_{i},y_{j})\in \overline{\Omega}^{h}\right\}.
\end{equation*}
We assume that on $\langle \widehat{U},\widetilde{U}\rangle$, the vector function $f(x_{i},y_{j},U)$ in (\ref{eq:2}), satisfy the constraints
\begin{equation}\label{eq:4}
  \frac{\partial f_{\alpha}(x_{i},y_{j},U)}{\partial u_{\alpha}}\leq c_{\alpha}(x_{i},y_{j}),\quad U\in \langle \widehat{U},\widetilde{U}\rangle,\quad (x_{i},y_{j})\in \overline{\Omega}^{h},\quad \alpha=1,2,
\end{equation}
\begin{equation}\label{eq:5}
-\frac{\partial f_{\alpha}(x_{i},y_{j},U)}{\partial u_{\alpha^{\prime}}}\geq0,\quad U\in \langle \widehat{U},\widetilde{U}\rangle,\quad (x_{i},y_{j})\in \overline{\Omega}^{h},\quad \alpha^{\prime}\neq \alpha,\quad \alpha=1,2,
\end{equation}
where $c_{\alpha}(x_{i},y_{j})$, $\alpha=1,2$, are non-negative bounded functions in $\overline{\Omega}^{h}$.
We say that the vector function $f(x_{i},y_{j},U)$ is quasimonotone nondecreasing on $\langle \widehat{U},\widetilde{U}\rangle$ if it satisfies (\ref{eq:5}).
\begin{remark}\label{remark2}
In this remark we discuss the mean-value theorem for vector-valued functions. Introduce the following notation:
\begin{equation}\label{eq:mvth}
\mathcal{F}_{\alpha}(x,y,u_{\alpha},u_{\alpha^{\prime}})=\left\{ \begin{array}{ll}
\mathcal{F}_{1}(x,y,u_{1},u_{2}),\quad \alpha=1,\\
\mathcal{F}_{2}(x,y,u_{1},u_{2}),\quad \alpha=2,
\end{array}\right.
\alpha\neq \alpha^{\prime}.
\end{equation}
Assume that $\mathcal{F}_{\alpha}(x,y,u_{\alpha},u_{\alpha^{\prime}})$, $\alpha=1,2$, are smooth functions, then we have
\begin{equation}\label{eq:mvth1}
   \mathcal{F}_{\alpha}(x,y,u_{\alpha},u_{\alpha^{\prime}})-\mathcal{F}_{\alpha}(x,y,w_{\alpha},u_{\alpha^{\prime}})=\frac{\partial \mathcal{F}_{\alpha}(h_{\alpha}, u_{\alpha^{\prime}})}{\partial u_{\alpha}}[u_{\alpha}-w_{\alpha}],
  \end{equation}
   \begin{equation*}
   \mathcal{F}_{\alpha}(x,y,u_{\alpha},u_{\alpha^{\prime}})-\mathcal{F}_{\alpha}(x,y,u_{\alpha},w_{\alpha^{\prime}})=\frac{\partial \mathcal{F}_{\alpha}(u_{\alpha},h_{\alpha^{\prime}})}{\partial u_{\alpha^{\prime}}}[u_{\alpha^{\prime}}-w_{\alpha^{\prime}}],
  \end{equation*}
  where $h_{\alpha}(x,y)$ lies between $u_{\alpha}(x,y)$ and $w_{\alpha}(x,y)$, and  $h_{\alpha^{\prime}}(x,y)$ lies between $u_{\alpha^{\prime}}(x,y)$ and $w_{\alpha^{\prime}}(x,y)$, $\alpha=1,2$.
\end{remark}
We introduce the notation
\begin{equation}\label{mono}
  \Gamma_{\alpha}(x_{i},y_{j},U)=c_{\alpha}(x_{i},y_{j})U_{\alpha}(x_{i},y_{j})-f_{\alpha}(x_{i},y_{j},U),\quad (x_{i},y_{j})\in \overline{\Omega}^{h} ,\ \alpha=1,2,
\end{equation}
where $c_{\alpha}(x_{i},y_{j})$, $\alpha=1,2$, are defined in (\ref{eq:4}), and give a monotone property of $\Gamma_{\alpha}$, $\alpha=1,2$.
\begin{lemma}\label{mono1}
Suppose that $U=(U_{1},U_{2})$ and $V=(V_{1},V_{2})$, are any functions in $\langle\widehat{U},\widetilde{U}\rangle$, where
 $U\geq V$, and
assume that (\ref{eq:4}) and (\ref{eq:5}) are satisfied. Then
\begin{equation}\label{mono2}
  \Gamma_{\alpha}(U)\geq \Gamma_{\alpha}(V),\quad \alpha=1,2,
\end{equation}
where $(x_{i},y_{j})$ is suppressed in (\ref{mono2}).
\end{lemma}
\begin{proof}
  From (\ref{mono}), we have
  \begin{eqnarray}\label{mono3}
  \Gamma_{\alpha}(U)-\Gamma _{\alpha}(V)&=&c_{\alpha}(x_{i},y_{j})(U_{\alpha}(x_{i},y_{j})-V_{\alpha}(x_{i},y_{j}))\\ \nonumber
 & -& \left[f_{\alpha}(x_{i},y_{j},U_{1},U_{2})-f_{\alpha}(x_{i},y_{j},V_{1},U_{2})\right]\\ \nonumber
 & -& \left[f_{\alpha}(x_{i},y_{j},V_{1},U_{2})-f_{\alpha}(x_{i},y_{j},V_{1},V_{2})\right]. \nonumber
\end{eqnarray}
For $\alpha=1$ in (\ref{mono3}), using the mean-value theorem (\ref{eq:mvth1}), we obtain
\begin{eqnarray*}
  \Gamma_{1}(U)-\Gamma _{1}(V)&=& \left(c_{1}(x_{i},y_{j})-\frac{\partial f_{1}(Q_{1},U_{2})}{\partial u_{1}}\right)(U_{1}-V_{1})\\ &-&\frac{\partial f_{1}(V_{1},Q_{2})}{\partial u_{2}}(U_{2}-V_{2}),
\end{eqnarray*}
where
$$
V_{\alpha}(x_{i},y_{j})\leq Q_{\alpha}(x_{i},y_{j})\leq U_{\alpha}(x_{i},y_{j}),\quad (x_{i},y_{j})\in \overline{\Omega}^{h},\quad \alpha=1,2.
$$
From here, (\ref{eq:4}), (\ref{eq:5}) and taking into account that $U_{\alpha}\geq V_{\alpha}$, $\alpha=1,2$, we conclude that
$$
 \Gamma_{1}(U)-\Gamma _{1}(V)\geq0.
 $$
 Similarly, we can prove that
 $$
 \Gamma_{2}(U)-\Gamma _{2}(V)\geq0.
 $$
\end{proof}
We introduce the linear version of problem (\ref{eq:2}) in the form
\begin{equation}\label{eq:lp}
 \mathcal{L}_{\alpha,ij}W_{\alpha}(x_{i},y_{j})+c^{*}_{\alpha}(x_{i},y_{j})W_{\alpha}(x_{i},y_{j})=\Phi_{\alpha}(x_{i},y_{j}),\quad (x_{i},y_{j})\in \Omega^{h},
\end{equation}
\begin{equation*}
  W(x_{i},y_{j})=g(x_{i},y_{j}),\quad (x_{i},y_{j})\in \partial\Omega^{h},\quad \alpha=1,2,
\end{equation*}
where $W(x_{i},y_{j})\equiv(W_{1}(x_{i},y_{j}),W_{2}(x_{i},y_{j}))$, $(x_{i},y_{j})\in \overline{\Omega}^{h}$, and $c^{*}_{\alpha}(x_{i},y_{j})$, $\alpha=1,2$, are non-negative bounded functions.  We formulate the maximum principle for the difference operator $\mathcal{L}_{\alpha,ij}+c^{*}_{\alpha}(x_{i},y_{j})$,
$(x_{i},y_{j})\in \Omega^{h}$, $\alpha=1,2$, and give an estimate of the solution to (\ref{eq:lp}).
\begin{lemma}\label{lm1}
 \begin{itemize}
\item[(i)] If $W_{\alpha}(x_{i},y_{j})$, $\alpha=1,2$, satisfy the conditions
  \begin{equation*}
  \mathcal{L}_{\alpha,ij}W_{\alpha}(x_{i},y_{j})+c^{*}_{\alpha}(x_{i},y_{j})W_{\alpha}(x_{i},y_{j})\geq0 \ (\leq 0),\quad (x_{i},y_{j})\in \Omega^{h},
  \end{equation*}
  \begin{equation*}
    W_{\alpha}(x_{i},y_{j})\geq0 \ (\leq 0),\quad (x_{i},y_{j})\in \partial \Omega^{h},\quad \alpha=1,2,
  \end{equation*}
  then $W_{\alpha}(x_{i},y_{j})\geq0 \ (\leq0),\quad (x_{i},y_{j}) \in \overline{\Omega}^{h}$, $\alpha=1,2$.
\item[(ii)]
 The following estimate of the solution to (\ref{eq:lp}) holds true
  \begin{equation}\label{eq:es}
   \|W_{\alpha}\|_{\overline{\Omega}^{h}}\leq \max \left\{\|g_{\alpha}\|_{\partial \Omega^{h}}, \|\Phi_{\alpha}/c^{*}_{\alpha}\|_{\Omega^{h}}\right\},
  \end{equation}
  where
  \begin{equation*}
    \|g_{\alpha}\|_{\partial\Omega^{h}}=\max_{(x_{i},y_{j})\in\partial\Omega^{h}}|g_{\alpha}(x_{i},y_{j})|,\quad \left\| \frac{ \Phi_{\alpha}}{c^{*}_{\alpha}}\right\|_{\Omega^{h}}=\max_{(x_{i},y_{j})\in\Omega^{h}}\left| \frac{\Phi_{\alpha}(x_{i},y_{j})}{c^{*}_{\alpha}(x_{i},y_{j})}\right|.
  \end{equation*}
\end{itemize}
\end{lemma}
The proof of the lemma can be found in \cite{Ab79}, \cite{sm2001}.
\section{Block monotone schemes}
Write down the difference scheme (\ref{eq:2}) at an interior mesh point $(x_{i},y_{j})\in \Omega^{h}$ in the form
\begin{eqnarray}\label{eq:nds}
d_{\alpha,ij}U_{\alpha,ij}-l_{\alpha,ij}U_{\alpha,i-1,j}&-&r_{\alpha,ij}U_{\alpha,i+1,j}-b_{\alpha,ij}U_{\alpha,i,j-1}\\&-&t_{\alpha,ij}U_{\alpha,i,j+1}= -f_{\alpha}(x_{i},y_{j},U_{1,ij},U_{2,ij})+G_{\alpha,ij}^{*},\nonumber
  \end{eqnarray}
\begin{equation*}
l_{\alpha,ij}=\frac{\varepsilon_{\alpha}}{h_{x}^{2}}+\frac{v_{\alpha}(x_{i},y_{j})}{h_{x}},\quad r_{\alpha,ij}=\frac{\varepsilon_{\alpha}}{h_{x}^{2}},
\end{equation*}
\begin{equation*}
b_{\alpha,ij}=\frac{\varepsilon_{\alpha}}{h_{y}^{2}}+\frac{v_{\alpha}(x_{i},y_{j})}{h_{y}},\quad t_{\alpha,ij}=\frac{\varepsilon_{\alpha}}{h_{y}^{2}},
\end{equation*}
\begin{equation*}
d_{\alpha,ij}=l_{\alpha,ij}+r_{\alpha,ij}+b_{\alpha,ij}+t_{\alpha,ij},\quad \alpha=1,2,
\end{equation*}
where $G_{\alpha,ij}^{*}$ is  associated with the boundary function $g_{\alpha}(x_{i},y_{j})$. Define vectors and diagonal matrices by
\begin{equation*}
  U_{\alpha,i}=(U_{\alpha,i,1}, \ldots, U_{\alpha,i,N_{y}-1})^{T},\quad G_{\alpha,i}^{*}=(G_{\alpha,i,1}^{*}, \ldots, G_{\alpha,i,N_{y}-1}^{*})^{T},
\end{equation*}
\begin{equation*}
  F_{\alpha,i}(U_{1,i},U_{2,i})=(f_{\alpha,i,1}(U_{1,i,1},U_{2,i,1}), \ldots, f_{\alpha,i,N_{y}-1}(U_{1,i,N_{y}-1},U_{2,i,N_{y}-1}))^{T},
 \end{equation*}
\begin{equation*}
  L_{\alpha,i}=\mbox{diag}(l_{\alpha,i,1},\ldots, l_{\alpha,i,N_{y}-1}),\ R_{\alpha,i}=\mbox{diag}(r_{\alpha,i,1},\ldots, r_{\alpha,i,N_{y}-1}),\ \alpha=1,2,
\end{equation*}
where $L_{\alpha,1}U_{\alpha,0}$ is included in $G_{\alpha,1}^{*}$, and  $R_{\alpha,N_{x}-1}U_{\alpha,N_{x}}$ is included in $G_{\alpha,N_{x}}^{*}$. Then the difference scheme (\ref{eq:2}) may be written in the form
\begin{equation}\label{eq:fds}
 A_{\alpha,i}U_{\alpha,i}-(L_{\alpha,i}U_{\alpha,i-1}+R_{\alpha,i}U_{\alpha,i+1})=-F_{\alpha,i}(U_{i})+G_{\alpha,i}^{*},
\end{equation}
\begin{equation*}
U_{i}=(U_{1,i},U_{2,i}),\quad i=1,2,\ldots,N_{x}-1,\quad \alpha=1,2,
\end{equation*}
with the tridiagonal matrix $A_{\alpha,i}$ in the form
$$
A_{\alpha,i}= \left[\begin{array}{ccccccc}
d_{\alpha,i,1}       &        &     -t_{\alpha,i,1}      &                 &         &         0\\   [1.5 ex]
-b_{\alpha,i,2}      &             & d_{\alpha,i,2}      &       -t_{\alpha,i,2}          &         &        \\    [1 ex]
              &               & \  \ddots   & \ \ddots        &   \ddots      &      \\    [1 ex]
              &                &           &\ \ - b_{\alpha,i,N_{y}-2}   &  \  d_{\alpha,i,N_{y}-2}     &  \  -t_{\alpha,i,N_{y}-2}      \\    [2.5 ex]
    0          &                &           &                   &     -b_{\alpha,i,N_{y}-1}                 & d_{\alpha,i,N_{y}-1}                       \\   
\end{array}\right]
.$$
Matrices $L_{\alpha,i}$ and $R_{\alpha,i}$ contain the coupling coefficients of a mesh point, respectively, to the mesh point of the left line and the mesh point of the right line.
\subsection{Block monotone Jacobi  method}
We present the  block monotone Jacobi  method for the difference scheme (\ref{eq:fds}). The upper $\{\overline{U}_{\alpha,i}^{(n)}\}$ and lower $\{\underline{U}_{\alpha,i}^{(n)}\}$, $\alpha=1,2$, sequences of solutions are calculated by the following block Jacobi iterative method
\begin{equation}\label{eq:6}
A_{\alpha,i}Z_{\alpha,i}^{(n)}+C_{\alpha,i}Z_{\alpha,i}^{(n)}=-\mathcal{K}_{\alpha,i}(U_{i}^{(n-1)}),\quad  i=1,2,\ldots,  N_{x}-1,\ \alpha=1,2,
\end{equation}
\begin{eqnarray*}
\mathcal{K}_{\alpha,i}(U_{i}^{(n-1)})&\equiv&A_{\alpha,i}U_{\alpha,i}^{(n-1)}-L_{\alpha,i}U_{\alpha,i-1}^{(n-1)}\\ &-&R_{\alpha,i}U_{\alpha,i+1}^{(n-1)} +F_{\alpha,i}(U_{i}^{(n-1)})-G_{\alpha,i}^{*},
\end{eqnarray*}
\begin{equation*}
Z_{\alpha,i}^{(n)}=\left\{ \begin{array}{ll}
g_{\alpha,i}-U_{\alpha,i}^{(0)}, \quad n=1  ,\\
\mathbf{0}\quad \quad \quad \quad ,\quad  n\geq2,
\end{array}\right.
\quad i=0,N_{x},
 \end{equation*}
\begin{equation*}
Z_{\alpha,i}^{(n)}=U_{\alpha,i}^{(n)}-U_{\alpha,i}^{(n-1)},
\end{equation*}
where $U_{i}^{(n-1)}=(U_{1,i}^{(n-1)},U_{2,i}^{(n-1)})$, and  $\mathcal{K}_{\alpha,i}(U_{i}^{(n-1)})$, $\alpha=1,2$, are the residuals of the difference equations (\ref{eq:fds}) on $U_{\alpha,i}^{(n-1)}$, $\alpha=1,2$, and $\mathbf{0}$ is the  $(N_{y}-1)\times 1$ zero vector. Matrix $C_{\alpha,i}$ is the diagonal matrix  $\mbox{diag}(c_{\alpha,i,1}, \ldots, c_{\alpha,i,N_{y}-1})$, where $c_{\alpha,i,j}$, $(i,j)\in \Omega^{h}$, $\alpha=1,2$, are defined in (\ref{eq:4}).
\begin{remark}\label{remark3}
 Similar to Remark \ref{remark2}, we discuss the mean-value theorem for vector-functions. Introduce the following notations:
\begin{equation}\label{eq:mvth2}
F_{\alpha,i}(U_{\alpha,i},U_{\alpha^{\prime},i})=\left\{ \begin{array}{ll}
F_{1,i}(U_{1,i},U_{2,i}),\quad \alpha=1,\\
F_{2,i}(U_{1,i},U_{2,i}),\quad \alpha=2,
\end{array}\right.
\alpha\neq \alpha^{\prime},\quad i=0,1,\ldots,N_{x}.
\end{equation}
Assume that $F_{\alpha,i}(U_{\alpha,i},U_{\alpha^{\prime},i})$, $i=0,1,\ldots,N_{x}$, $\alpha\neq \alpha^{\prime}$, $\alpha=1,2$, are smooth functions, then we have
  \begin{equation}\label{eq:mvth3}
  F_{\alpha,i}(U_{\alpha,i},U_{\alpha^{\prime},i})-F_{\alpha,i}(V_{\alpha,i},U_{\alpha^{\prime},i})= \frac{\partial F_{\alpha,i}(Y_{\alpha,i}, U_{\alpha^{\prime},i})}{\partial u_{\alpha}}[U_{\alpha,i}-V_{\alpha,i}],
 \end{equation}
   \begin{equation*}
   F_{\alpha,i}(U_{\alpha,i},U_{\alpha^{\prime},i})-F_{\alpha,i}(U_{\alpha,i},V_{\alpha^{\prime},i})=\frac{\partial F_{\alpha,i}(U_{\alpha,i},Y_{\alpha^{\prime},i})}{\partial  u_{\alpha^{\prime}}}[U_{\alpha^{\prime},i}-V_{\alpha^{\prime},i}],
  \end{equation*}
  where $Y_{\alpha,i}$ lies between $U_{\alpha,i}$ and $V_{\alpha,i}$, and  $Y_{\alpha^{\prime},i}$ lies between $U_{\alpha^{\prime},i}$ and $V_{\alpha^{\prime},i}$, $i=0,1,\ldots,N_{x}$, $\alpha=1,2$.
  The partial derivatives $\frac{\partial F_{\alpha,i}}{\partial u_{\alpha}}$ and $\frac{\partial F_{\alpha,i}}{\partial u_{\alpha^{\prime}}}$, are the diagonal matrices
 \begin{eqnarray*}
&& \frac{\partial F_{\alpha,i}}{\partial u_{\alpha}}= \\ && \mbox{diag} \left(\frac{\partial f_{\alpha,i,1}}{\partial u_{\alpha}}(Y_{\alpha,i,1}, U_{\alpha^{\prime},i,1}), \ldots, \frac{\partial f_{\alpha,i,N_{y}-1}}{\partial u_{\alpha}}(Y_{\alpha,i,N_{y}-1}, U_{\alpha^{\prime},i,N_{y}-1})\right),
 \end{eqnarray*}
\begin{eqnarray*}
&& \frac{\partial F_{\alpha,i}}{\partial u_{\alpha^{\prime}}}= \\ && \mbox{diag} \left(\frac{\partial f_{\alpha,i,1}}{\partial u_{\alpha^{\prime}}}( U_{\alpha,i,1},Y_{\alpha^{\prime},i,1}), \ldots, \frac{\partial f_{\alpha,i,N_{y}-1}}{\partial u_{\alpha^{\prime}}}(U_{\alpha,i,N_{y}-1},Y_{\alpha^{\prime},i,N_{y}-1})\right).
 \end{eqnarray*}
\end{remark}
\begin{theorem}\label{mc}
Assume that $f_{\alpha}(x_{i},y_{j},U)$, $\alpha=1,2$, satisfy (\ref{eq:4}) and (\ref{eq:5}). Let $\widetilde{U}=(\widetilde{U}_{1},\widetilde{U}_{2})$ and $\widehat{U}=(\widehat{U}_{1},\widehat{U}_{2})$ be, respectively, ordered upper and lower solutions of (\ref{eq:2}). Then the upper $\{\overline{U}_{\alpha,i}^{(n)}\}$ and lower $\{\underline{U}_{\alpha,i}^{(n)}\}$, $i=0,1,\ldots,N_{x}$, $\alpha=1,2$, sequences  generated by (\ref{eq:6}), with $\overline{U}^{(0)}=\widetilde{U}$ and $\underline{U}^{(0)}=\widehat{U}$, converge monotonically, respectively, from above to a maximal solution $\overline{U}$ and from below to a minimal solution $\underline{U}$, such that,
\begin{equation}\label{eq:m}
   \underline{U}_{\alpha,i}^{(n-1)}\leq \underline{U}_{\alpha,i}^{(n)} \leq \underline{U}_{\alpha,i}\leq \overline{U}_{\alpha,i}\leq \overline{U}_{\alpha,i}^{(n)}\leq \overline{U}_{\alpha,i}^{(n-1)},\ i=0,1, \ldots, N_{x}, \ \alpha=1,2.
  \end{equation}
  If   $S=(S_{1},S_{2})$ is any other solution  in $\langle \widehat{U},\widetilde{U}\rangle$, then
  \begin{equation}\label{eq:un}
  \underline{U}\leq S \leq \overline{U}, \quad \mbox{in} \quad \overline{\Omega}^{h}.
 \end{equation}
\end{theorem}
\begin{proof}
Since $\overline{U}^{(0)}$ is an initial upper solution (\ref{eq:3}), from (\ref{eq:6}), we have
\begin{equation*}
 (A_{\alpha,i}+C_{\alpha,i})\overline{Z}_{\alpha,i}^{(1)}\leq\mathbf{0},\quad i=1,2,\ldots, N_{x}-1,
\end{equation*}
\begin{equation*}
  \overline{Z}_{\alpha,i}^{(1)}\leq\mathbf{0},\quad i=0,N_{x},\quad \alpha=1,2.
\end{equation*}
Taking into account that $d_{\alpha,ij}>0$, $b_{\alpha,ij}$, $t_{\alpha,ij}\geq0$, $\alpha=1,2$, in (\ref{eq:nds}) and $A_{\alpha,i}$ are strictly diagonal dominant matrix, we conclude that  $A_{\alpha,i}$, $i=1,2,\ldots,N_{x}-1$, $\alpha=1,2$, are $M$-matrices  and $A_{\alpha,i}^{-1}\geq\emph{O}$ \cite{Varga2000},   which leads to $(A_{\alpha,i}+C_{\alpha,i})^{-1}\geq\emph{O}$, where $\emph{O}$ is the ($N_{y}-1)\times (N_{y}-1$) null  matrix. From here, we obtain
\begin{equation}\label{eq:7}
\overline{Z}_{\alpha,i}^{(1)}\leq \mathbf{0},\quad i=0,1 \ldots, N_{x},\quad \alpha=1,2.
\end{equation}
Similarly, we prove that
\begin{equation}\label{eq:8}
\underline{Z}_{\alpha,i}^{(1)}\geq\mathbf{0},\quad i=0,1, \ldots, N_{x},\quad \alpha=1,2.
\end{equation}
We now prove that
\begin{equation}\label{eq:1,1}
  \underline{U}_{\alpha,i}^{(1)}\leq \overline{U}_{\alpha,i}^{(1)},\quad i=0,1,\ldots,N_{x},\quad \alpha=1,2.
\end{equation}
 Letting $W_{\alpha,i}^{(n)}=\overline{U}_{\alpha,i}^{(n)}-\underline{U}_{\alpha,i}^{(n)}$, from (\ref{eq:6}) for $\alpha=1$, we have
\begin{eqnarray*}
  (A_{1,i}+C_{1,i})W_{1,i}^{(1)}&=&C_{1,i}W_{1,i}^{(0)}+L_{1,i}W_{1,i-1}^{(0)}+R_{1,i}W_{1,i+1}^{(0)}\\ &-&\left[F_{1,i}(\overline{U}_{1,i}^{(0)},\overline{U}_{2,i}^{(0)})-F_{1,i}(\underline{U}_{1,i}^{(0)},\overline{U}_{2,i}^{(0)})\right]\\
  &-& \left[F_{1,i}(\underline{U}_{1,i}^{(0)},\overline{U}_{2,i}^{(0)})-F_{1,i}(\underline{U}_{1,i}^{(0)},\underline{U}_{2,i}^{(0)})\right],
\end{eqnarray*}
\begin{equation*}
i=1,2, \ldots, N_{x}-1,\quad  W_{1,i}^{(1)}=\mathbf{0},\quad i=0, N_{x}.
\end{equation*}
By the mean-value theorem (\ref{eq:mvth3}), we have
\begin{eqnarray*}
&&F_{1,i}(\overline{U}_{1,i}^{(0)},\overline{U}_{2,i}^{(0)})-F_{1,i}(\underline{U}_{1,i}^{(0)},\overline{U}_{2,i}^{(0)})=\\
&&\left(f_{1,i,1}(\overline{U}^{(0)}_{1,i,1},\overline{U}^{(0)}_{2,i,1}),\ldots,f_{1,i,N_{y}-1}(\overline{U}^{(0)}_{1,i,N_{y}-1},\overline{U}^{(0)}_{2,i,N_{y}-1})\right)^{T}-\\
&&\left(f_{1,i,1}(\underline{U}^{(0)}_{1,i,1},\overline{U}^{(0)}_{2,i,1}),\ldots,f_{1,i,N_{y}-1}(\underline{U}^{(0)}_{1,i,N_{y}-1},\overline{U}^{(0)}_{2,i,N_{y}-1})\right)^{T}=\
\end{eqnarray*}
$$
\frac{\partial F_{1,i}(Q_{1,i}^{(0)},\overline{U}_{2,i}^{(0)})}{\partial u_{1}}\left[\overline{U}_{1,i}^{(0)}-\underline{U}_{1,i}^{(0)}\right],
$$
\begin{eqnarray*}
&&F_{1,i}(\underline{U}_{1,i}^{(0)},\overline{U}_{2,i}^{(0)})-F_{1,i}(\underline{U}_{1,i}^{(0)},\underline{U}_{2,i}^{(0)})=\\
&&\left(f_{1,i,1}(\underline{U}^{(0)}_{1,i,1},\overline{U}^{(0)}_{2,i,1}),\ldots,f_{1,i,N_{y}-1}(\underline{U}^{(0)}_{1,i,N_{y}-1},\overline{U}^{(0)}_{2,i,N_{y}-1})\right)^{T}-\\
&&\left(f_{1,i,1}(\underline{U}^{(0)}_{1,i,1},\underline{U}^{(0)}_{2,i,1}),\ldots,f_{1,i,N_{y}-1}(\underline{U}^{(0)}_{1,i,N_{y}-1},\underline{U}^{(0)}_{2,i,N_{y}-1})\right)^{T}=\
\end{eqnarray*}
$$
\frac{\partial F_{1,i}(\underline{U}_{1,i}^{(0)},Q_{2,i}^{(0)})}{\partial u_{2}}\left[\overline{U}_{2,i}^{(0)}-\underline{U}_{2,i}^{(0)}\right],
$$
where $\underline{U}_{\alpha,i}^{(0)}\leq Q_{\alpha,i}^{(0)}\leq \overline{U}_{\alpha,i}^{(0)}$, $i=0,1,\ldots,N_{x}$, $\alpha=1,2$, and
\begin{eqnarray*}
&&\frac{\partial F_{1,i}(Q_{1,i}^{(0)},\overline{U}_{2,i}^{(0)})}{\partial u_{1}}=\\ &&\mbox{diag} \left(\frac{\partial f_{1,i,1}}{\partial u_{1}}(Q_{1,i,1}^{(0)},\overline{U}_{2,i,1}^{(0)}),\ldots,\frac{\partial f_{1,i,N_{y}-1}}{\partial u_{1}}(Q_{1,i,N_{y}-1}^{(0)},\overline{U}_{2,i,N_{y}-1}^{(0)}) \right),
\end{eqnarray*}
\begin{eqnarray*}
&&\frac{\partial F_{1,i}(\overline{U}_{1,i}^{(0)},Q_{2,i}^{(0)})}{\partial u_{2}}= \\ &&\mbox{diag} \left(\frac{\partial f_{1,i,1}}{\partial u_{2}}(\overline{U}_{1,i,1}^{(0)},Q_{2,i,1}^{(0)}),\ldots,\frac{\partial f_{1,i,N_{y}-1}}{\partial u_{2}}(\overline{U}_{1,i,N_{y}-1}^{(0)},Q_{2,i,N_{y}-1}^{(0)})\right).
\end{eqnarray*}
From here, (\ref{eq:4}), (\ref{eq:5}) and taking into account that  $\underline{U}^{(0)}\leq \overline{U}^{(0)}$ in $\overline{\Omega}^{h}$, $L_{\alpha,i}\geq \emph{O}$, $R_{\alpha,i}\geq \emph{O}$,
$i=1,2,\ldots,N_{x}-1$, $\alpha=1,2$, we obtain
$$
(A_{1,i}+C_{1,i})W_{1,i}^{(1)}\geq\mathbf{0},\quad i=1,2, \ldots, N_{x}-1,
$$
\begin{equation*}
W_{1,i}^{(1)}=\mathbf{0},\quad i=0, N_{x}.
\end{equation*}
From here and  $(A_{\alpha,i}+C_{\alpha,i})^{-1}\geq\emph{O}$, $i=1,2, \ldots, N_{x}-1$, $\alpha=1,2$, we obtain
\begin{equation*}
 W_{1,i}^{(1)}\geq\mathbf{0},\quad i=0,1, \ldots, N_{x}.
 \end{equation*}
Similarly, we can prove that
\begin{equation*}
 W_{2,i}^{(1)}\geq\mathbf{0},\quad i=0,1,\ldots, N_{x}.
 \end{equation*}
We now prove that $\overline{U}_{\alpha,i}^{(1)}$ and $\underline{U}_{\alpha,i}^{(1)}$, $i=0,1, \ldots, N_{x}$, $\alpha=1,2$, are, respectively, upper and lower solutions to (\ref{eq:2}).  For $\alpha=1$, from (\ref{eq:6}) and using the mean-value theorem (\ref{eq:mvth3}), we obtain
\begin{eqnarray}\label{eq:R}
\mathcal{K}_{1,i}(\overline{U}_{i}^{(1)})&=&-\left(C_{1,i}-\frac{\partial F_{1,i}(\overline{E}^{(1)}_{1,i},\overline{U}^{(0)}_{2,i})}{\partial u_{1}}\right) \overline{Z}_{1,i}^{(1)}\\&+&\frac{\partial F_{1,i}(\overline{U}^{(0)}_{1,i},\overline{E}^{(1)}_{2,i})}{\partial u_{2}}\overline{Z}_{2,i}^{(1)}\nonumber
-L_{1,i}\overline{Z}_{1,i-1}^{(1)}-R_{1,i}\overline{Z}_{1,i+1}^{(1)},
\end{eqnarray}
$$
i=1,2, \ldots, N_{x}-1,
$$
where
 $$
\overline{U}^{(1)}_{\alpha,i}\leq\overline{E}^{(1)}_{\alpha,i}\leq \overline{U}^{(0)}_{\alpha,i},\quad i=0,1,\ldots,N_{x},\quad \alpha=1,2.
$$
 From (\ref{eq:7}), (\ref{eq:8}) and (\ref{eq:1,1}), we conclude that $\frac{\partial F_{1,i}}{\partial u_{1}}$ and $\frac{\partial F_{1,i}}{\partial u_{2}}$ satisfy
 (\ref{eq:4}) and (\ref{eq:5}). From (\ref{eq:4}), (\ref{eq:5}), (\ref{eq:7}) and taking into account that $L_{\alpha,i}\geq\emph{O}$, $R_{\alpha,i}\geq \emph{O}$, $i=1,2,\ldots,N_{x}-1$, $\alpha=1,2 $, we conclude that
 \begin{equation}\label{eq:k}
 \mathcal{K}_{1,i}(\overline{U}^{(1)}_{i})\geq \mathbf{0},\quad i=1,2,\ldots,N_{x}-1.
 \end{equation}
 Similarly, we conclude that
  \begin{equation}\label{eq:k1}
 \mathcal{K}_{2,i}(\overline{U}^{(1)}_{i})\geq \mathbf{0}, \quad i=1,2,\ldots,N_{x}-1.
 \end{equation}
 From (\ref{eq:3}), (\ref{eq:k}) and (\ref{eq:k1}), we conclude that $\overline{U}^{(1)}_{i}$, $i=0,1,\ldots,N_{x}$, is an upper solution to (\ref{eq:2}). By following a similar argument, we prove that
 $$
 \mathcal{K}_{1,i}(\underline{U}^{(1)}_{i})\leq \mathbf{0},\quad \mathcal{K}_{2,i}(\underline{U}^{(1)}_{i})\leq \mathbf{0}, \quad i=1,2,\ldots,N_{x}-1.
 $$
 By induction on $n$, we can prove that the sequences  $\{\overline{U}^{(n)}_{\alpha,i}\}$, $\{\underline{U}^{(n)}_{\alpha,i}\}$, $i=0,1,\ldots,N_{x}$, $\alpha=1,2$, are, respectively, monotone decreasing upper and monotone
 increasing lower sequences of solutions.

Now we prove that the limiting functions of the upper $\{\overline{U}_{\alpha,i}^{(n)}\}$ and lower $\{\underline{U}_{\alpha,i}^{(n)}\}$, $i=0,1,\ldots,N_{x}$, $\alpha=1,2$, sequences are, respectively,  maximal and minimal solutions  of (\ref{eq:2}).
From (\ref{eq:m}), we conclude that $\lim \overline{U}_{\alpha,i}^{(n)}=\overline{U}_{\alpha,i}$, $i=0,1,\ldots,N_{x}$, $\alpha=1,2$,   as $n\rightarrow \infty$ exists and
  \begin{equation}\label{eq:9}
  \lim_{n\rightarrow \infty}\overline{Z}_{\alpha,i}^{(n)}=\mathbf{0},\quad i=0,1, \ldots, N_{x},\quad \alpha=1,2.
  \end{equation}
  Similar to (\ref{eq:R}), we have
  \begin{eqnarray}\label{eq:Rn}
\mathcal{K}_{\alpha,i}(\overline{U}_{i}^{(n)})&=&-\left(C_{\alpha,i}-\frac{\partial F_{\alpha,i}(\overline{E}^{(n)}_{\alpha,i},\overline{U}^{(n-1)}_{\alpha^{\prime},i})}{\partial u_{\alpha}}\right) \overline{Z}_{\alpha,i}^{(n)} \\ &+&  \frac{\partial F_{\alpha,i}(\overline{U}^{(n-1)}_{\alpha,i},\overline{E}^{(n)}_{\alpha^{\prime},i})}{\partial u_{\alpha^{\prime}}}\overline{Z}_{\alpha^{\prime},i}^{(n)}
-L_{\alpha,i}\overline{Z}_{\alpha,i-1}^{(n)}-R_{\alpha,i}\overline{Z}_{\alpha,i+1}^{(n)}, \nonumber
\end{eqnarray}
$$
  i=1,2, \ldots, N_{x}-1,\quad \alpha=1,2,
$$
where
$$\overline{U}_{\alpha,i}^{(n)}\leq \overline{E}^{(n)}_{\alpha,i}\leq \overline{U}_{\alpha,i}^{(n-1)},\quad i=0,1,\ldots,N_{x},\quad \alpha=1,2.
$$
By taking the limit of both sides of  (\ref{eq:Rn}) and using (\ref{eq:9}), it follows that
$$\mathcal{K}_{\alpha,i}(\overline{U}_{i})=\mathbf{0},\quad i=1,2,\ldots,N_{x}-1,\quad \alpha=1,2,$$
which means that $\overline{U}_{i}$, $i=0,1,\ldots,N_{x}$, is a maximal solution to the nonlinear difference scheme (\ref{eq:2}). In similar manner, we can prove that
 $$\mathcal{K}_{\alpha,i}(\underline{U}_{i}^{(n)})=\mathbf{0},\quad i=1,2,\ldots,N_{x}-1,\quad \alpha=1,2,$$
 which means that $(\underline{U}_{i})$, $i=0,1,\ldots,N_{x}$, is a minimal solution to the nonlinear difference scheme (\ref{eq:2}).

 Now, we prove (\ref{eq:un}). We assume that  $S=(S_{1},S_{2})$ is another solution in $\langle\widehat{U},\widetilde{U}\rangle$.  We consider the sector $\langle S, \widetilde{U}\rangle$, which means that we  treat  $S$ as a lower solution. Since $\{\underline{S}^{(n)}\}=\{S\}$ is a constant sequence for all $n$, then from (\ref{eq:m}), we conclude that $\widetilde{U}_{\alpha,i}\geq S_{\alpha,i}$, $i=0,1,\ldots,N_{x}$, $\alpha=1,2$. Now,  we consider the sector $ \langle \widehat{U},S\rangle$, which means that we treat  $S$ as an upper solution. Similarly,  since $\{\overline{S}^{(n)}\}=\{S\}$ is a constant sequence for all n, then from (\ref{eq:m}), we conclude that $\widehat{U}_{\alpha,i}\leq S_{\alpha,i}$, $i=0,1,\ldots,N_{x}$, $\alpha=1,2$. Thus, we prove (\ref{eq:un}).
 \end{proof}
\subsection{Block monotone Gauss--Seidel  method}
We introduce the block monotone Gauss--Seidel  method for the nonlinear difference scheme (\ref{eq:fds}). The iterative sequences $\{\overline{U}_{\alpha,i}^{(n)}\}$, $\{\underline{U}^{(n)}_{\alpha,i}\}$,
$i=0,1,\ldots,N_{x}$, $\alpha=1,2$, are calculated by using the block Gauss--Seidel iterative method
\begin{equation}\label{eq:10}
A_{\alpha,i}Z_{\alpha,i}^{(n)}-L_{\alpha,i}Z^{(n)}_{\alpha,i-1}+C_{\alpha,i}Z_{\alpha,i}^{(n)}=-\mathcal{K}_{\alpha,i}(U_{i}^{(n-1)}),\quad  i=1,\ldots,  N_{x}-1,
\end{equation}
\begin{eqnarray*}
\mathcal{K}_{\alpha,i}(U_{i}^{(n-1)})&=&A_{\alpha,i}U_{\alpha,i}^{(n-1)}-L_{\alpha,i}U_{\alpha,i-1}^{(n-1)}\\ &-&R_{\alpha,i}U_{\alpha,i+1}^{(n-1)} +F_{\alpha,i}(U_{i}^{(n-1)})-G_{\alpha,i}^{*},
\end{eqnarray*}
\begin{equation*}
Z_{\alpha,i}^{(n)}=\left\{ \begin{array}{ll}
g_{\alpha,i}-U_{\alpha,i}^{(0)}, \quad n=1  ,\\
\mathbf{0},\quad \quad \quad \quad n\geq2,
\end{array}\right.
\quad i=0,N_{x},
 \end{equation*}
\begin{equation*}
Z_{\alpha,i}^{(n)}=U_{\alpha,i}^{(n)}-U_{\alpha,i}^{(n-1)},\quad \alpha=1,2,
\end{equation*}
where $U_{i}^{(n)}=(U_{1,i}^{(n)},U_{2,i}^{(n)})$, and  $\mathcal{K}_{\alpha,i}(U_{i}^{(n-1)})$, $\alpha=1,2$, are the residuals of the difference equations (\ref{eq:fds}) on $U_{\alpha,i}^{(n-1)}$, $\alpha=1,2$, and $\mathbf{0}$ is the  $(N_{y}-1)\times 1$ zero vector. Matrix $C_{\alpha,i}$ is the diagonal matrix  $\mbox{diag}(c_{\alpha,i,1}, \ldots, c_{\alpha,i,N_{y}-1})$, where $c_{\alpha,i,j}$, $(i,j)\in \Omega^{h}$, $\alpha=1,2$, are defined in (\ref{eq:4}).
\begin{remark}
In this remark, we discuss the implementation of the block Gauss-Seidel iterative method.
\begin{itemize}
  \item [(i)] If $v_{\alpha}(x,y)\equiv0$, $ \alpha=1,2$, in (\ref{cont1a}), then we can start the block Gauss-Seidel method implementation from either $i=0$ or $i=N_{x}$, that is, we start either from the left or the right.
  \item [(ii)] If $v_{\alpha}(x,y)\geq0$, $\alpha=1,2$, in (\ref{cont1a}), then we  start the block Gauss-Seidel method implementation from $i=0$, that is, we start from the left.
  \item [(iii)] If $v_{\alpha}(x,y)\leq0$, $\alpha=1,2$, in (\ref{cont1a}), then we  start the block Gauss-Seidel method implementation from $i=N_{x}$, that is, we start from the right. In this case, the block
  Gauss-Seidel iterative method (\ref{eq:10}) can be written in the form
  \begin{equation*}
A_{\alpha,i}Z_{\alpha,i}^{(n)}-R_{\alpha,i}Z^{(n)}_{\alpha,i+1}+C_{\alpha,i}Z_{\alpha,i}^{(n)}=-\mathcal{K}_{\alpha,i}(U_{i}^{(n-1)}),\quad  i=1,\ldots,  N_{x}-1,
\end{equation*}
\begin{eqnarray*}
\mathcal{K}_{\alpha,i}(U_{i}^{(n-1)})&=&A_{\alpha,i}U_{\alpha,i}^{(n-1)}-L_{\alpha,i}U_{\alpha,i-1}^{(n-1)}\\ &-&R_{\alpha,i}U_{\alpha,i+1}^{(n-1)} +F_{\alpha,i}(U_{i}^{(n-1)})-G_{\alpha,i}^{*},
\end{eqnarray*}
\begin{equation*}
Z_{\alpha,i}^{(n)}=\left\{ \begin{array}{ll}
g_{\alpha,i}-U_{\alpha,i}^{(0)}, \quad n=1  ,\\
\mathbf{0},\quad \quad \quad \quad n\geq2,
\end{array}\right.
\quad i=0,N_{x},\quad \alpha=1,2.
 \end{equation*}
  \end{itemize}
\end{remark}
\begin{theorem}\label{mc1}
Assume that $f_{\alpha}(x_{i},y_{j},U)$, $\alpha=1,2$, satisfy (\ref{eq:4}) and (\ref{eq:5}). Let $\widetilde{U}=(\widetilde{U}_{1},\widetilde{U}_{2})$ and $\widehat{U}=(\widehat{U}_{1},\widehat{U}_{2})$ be, respectively, ordered upper and lower solutions of (\ref{eq:2}). Then the upper $\{\overline{U}_{\alpha,i}^{(n)}\}$ and lower $\{\underline{U}_{\alpha,i}^{(n)}\}$, $i=0,1,\ldots,N_{x}$, $\alpha=1,2$, sequences  generated by (\ref{eq:10}), with $\overline{U}^{(0)}=\widetilde{U}$ and $\underline{U}^{(0)}=\widehat{U}$, converge monotonically, respectively, from above to a maximal solution $\overline{U}$ and from below to a minimal solution $\underline{U}$, such that,
\begin{equation}\label{eq:m1}
   \underline{U}_{\alpha,i}^{(n-1)}\leq \underline{U}_{\alpha,i}^{(n)} \leq \underline{U}_{\alpha,i}\leq \overline{U}_{\alpha,i}\leq \overline{U}_{\alpha,i}^{(n)}\leq \overline{U}_{\alpha,i}^{(n-1)},\ i=0,1, \ldots, N_{x}, \ \alpha=1,2.
  \end{equation}
  If   $S=(S_{1},S_{2})$ is any other solution  in $\langle \widehat{U},\widetilde{U}\rangle$, then
  \begin{equation}\label{eq:un1}
  \underline{U}\leq S\leq \overline{U},\quad \mbox{in} \   \overline{\Omega}^{h}.
 \end{equation}
\end{theorem}
\begin{proof}
Since $\overline{U}^{(0)}$ is an initial upper solution (\ref{eq:3}), from (\ref{eq:10}), we have
\begin{equation}\label{eq:11}
  A_{\alpha,i}\overline{Z}_{\alpha,i}^{(1)}-L_{\alpha,i}\overline{Z}_{\alpha,i-1}^{(1)}+C_{\alpha,i}\overline{Z}_{\alpha,i}^{(1)}=-\mathcal{K}_{\alpha,i}(\overline{U}_{i}^{(0)}),\ i=1,2,\ldots, N_{x}-1,
\end{equation}
\begin{equation*}
\overline{Z}_{\alpha,i}^{(1)}\leq0,\quad i=0,N_{x},\quad \alpha=1,2.
\end{equation*}
Taking into account that $L_{\alpha,i}\geq\emph{O}$, $(A_{\alpha,i}+C_{\alpha,i})^{-1}\geq\emph{O}$,\quad for $i=1$ in (\ref{eq:11}) and $\overline{Z}_{\alpha,0}^{(1)}\leq\mathbf{0}$, we conclude that $\overline{Z}_{\alpha,1}^{(1)}\leq\mathbf{0}$, $\alpha=1,2$. For $i=2$ in (\ref{eq:11}), using $L_{\alpha,2}\geq\emph{O}$ and $\overline{Z}_{\alpha,1}^{(1)}\leq\mathbf{0}$, we obtain $\overline{Z}_{\alpha,2}^{(1)}\leq\mathbf{0}$, $\alpha=1,2$. Thus, by induction on $i$, we prove that
\begin{equation}\label{eq:12}
\overline{Z}_{\alpha,i}^{(1)}\leq\mathbf{0},\quad i=0,1,\ldots,N_{x},\quad \alpha=1,2.
\end{equation}
Similarly, we can prove that
\begin{equation}\label{eq:13}
\underline{Z}_{\alpha,i}^{(1)}\geq\mathbf{0},\quad i=0,1,\ldots,N_{x},\quad \alpha=1,2.
\end{equation}
We now prove that
\begin{equation}\label{eq:14}
 \underline{U}_{\alpha,i}^{(1)}\leq \overline{U}_{\alpha,i}^{(1)},\quad i=0,1,\ldots,N_{x},\quad \alpha=1,2.
\end{equation}
Letting $W_{\alpha,i}^{(n)}=\overline{U}_{\alpha,i}^{(n)}-\underline{U}_{\alpha,i}^{(n)}$, $i=0,1,\ldots,N_{x}$, $\alpha=1,2$, from (\ref{eq:10}) for $\alpha=1$, we have
\begin{eqnarray}\label{eq:15}
  A_{1,i}W_{1,i}^{(1)}-L_{1,i}W_{1,i-1}^{(1)}+C_{1,i}W_{1,i}^{(1)}&=& C_{1,i}W_{1,i}^{(0)}+R_{1,i}W_{1,i+1}^{(0)} \\
  &-& \left[F_{1,i}(\overline{U}_{1,i}^{(0)},\overline{U}_{2,i}^{(0)})-F_{1,i}(\underline{U}_{1,i}^{(0)},\overline{U}_{2,i}^{(0)})\right] \nonumber \\
  &-& \left[F_{1,i}(\underline{U}_{1,i}^{(0)},\overline{U}_{2,i}^{(0)})-F_{1,i}(\underline{U}_{1,i}^{(0)},\underline{U}_{2,i}^{(0)})\right] \nonumber,
\end{eqnarray}
\begin{equation*}
i=1,2,\ldots,N_{x}-1,\quad   W_{1,i}^{(1)}=\mathbf{0},\quad i=0,N_{x}.
\end{equation*}
By the mean-value theorem (\ref{eq:mvth3}), we have
\begin{eqnarray*}
&&F_{1,i}(\overline{U}_{1,i}^{(0)},\overline{U}_{2,i}^{(0)})-F_{1,i}(\underline{U}_{1,i}^{(0)},\overline{U}_{2,i}^{(0)})=\\
&&\left(f_{1,i,1}(\overline{U}^{(0)}_{1,i,1},\overline{U}^{(0)}_{2,i,1}),\ldots,f_{1,i,N_{y}-1}(\overline{U}^{(0)}_{1,i,N_{y}-1},\overline{U}^{(0)}_{2,i,N_{y}-1})\right)^{T}-\\
&&\left(f_{1,i,1}(\underline{U}^{(0)}_{1,i,1},\overline{U}^{(0)}_{2,i,1}),\ldots,f_{1,i,N_{y}-1}(\underline{U}^{(0)}_{1,i,N_{y}-1},\overline{U}^{(0)}_{2,i,N_{y}-1})\right)^{T}=\
\end{eqnarray*}
$$
\frac{\partial F_{1,i}(Q_{1,i}^{(0)},\overline{U}_{2,i}^{(0)})}{\partial u_{1}}\left[\overline{U}_{1,i}^{(0)}-\underline{U}_{1,i}^{(0)}\right],
$$
\begin{eqnarray*}
&&F_{1,i}(\underline{U}_{1,i}^{(0)},\overline{U}_{2,i}^{(0)})-F_{1,i}(\underline{U}_{1,i}^{(0)},\underline{U}_{2,i}^{(0)})=\\
&&\left(f_{1,i,1}(\underline{U}^{(0)}_{1,i,1},\overline{U}^{(0)}_{2,i,1}),\ldots,f_{1,i,N_{y}-1}(\underline{U}^{(0)}_{1,i,N_{y}-1},\overline{U}^{(0)}_{2,i,N_{y}-1})\right)^{T}-\\
&&\left(f_{1,i,1}(\underline{U}^{(0)}_{1,i,1},\underline{U}^{(0)}_{2,i,1}),\ldots,f_{1,i,N_{y}-1}(\underline{U}^{(0)}_{1,i,N_{y}-1},\underline{U}^{(0)}_{2,i,N_{y}-1})\right)^{T}=\
\end{eqnarray*}
$$
\frac{\partial F_{1,i}(\underline{U}_{1,i}^{(0)},Q_{2,i}^{(0)})}{\partial u_{2}}\left[\overline{U}_{2,i}^{(0)}-\underline{U}_{2,i}^{(0)}\right],
$$
where $\underline{U}_{\alpha,i}^{(0)}\leq Q_{\alpha,i}^{(0)}\leq \overline{U}_{\alpha,i}^{(0)}$, $i=0,1,\ldots,N_{x}$, $\alpha=1,2$, and
\begin{eqnarray*}
&&\frac{\partial F_{1,i}(Q_{1,i}^{(0)},\overline{U}_{2,i}^{(0)})}{\partial u_{1}}=\\ &&\mbox{diag} \left(\frac{\partial f_{1,i,1}}{\partial u_{1}}(Q_{1,i,1}^{(0)},\overline{U}_{2,i,1}^{(0)}),\ldots,\frac{\partial f_{1,i,N_{y}-1}}{\partial u_{1}}(Q_{1,i,N_{y}-1}^{(0)},\overline{U}_{2,i,N_{y}-1}^{(0)}) \right),
\end{eqnarray*}
\begin{eqnarray*}
&&\frac{\partial F_{1,i}(\overline{U}_{1,i}^{(0)},Q_{2,i}^{(0)})}{\partial u_{2}}= \\ &&\mbox{diag} \left(\frac{\partial f_{1,i,1}}{\partial u_{2}}(\overline{U}_{1,i,1}^{(0)},Q_{2,i,1}^{(0)}),\ldots,\frac{\partial f_{1,i,N_{y}-1}}{\partial u_{2}}(\overline{U}_{1,i,N_{y}-1}^{(0)},Q_{2,i,N_{y}-1}^{(0)})\right).
\end{eqnarray*}
From here, we conclude that $\frac{\partial F_{1,i}}{\partial u_{1}}$, $\frac{\partial F_{1,i}}{\partial u_{2}}$ satisfy (\ref{eq:4}) and (\ref{eq:5}). From here and (\ref{eq:15}), we have
 \begin{eqnarray}\label{eq:16}
 A_{1,i}W_{1,i}^{(1)}-L_{1,i}W_{1,i-1}^{(1)}+C_{1,i}W_{1,i}^{(1)}&=&\left(C_{1,i}-\frac{\partial F_{1,i}}{\partial u_{1}}\right)W_{1,i}^{(0)}\\
 &-& \frac{\partial F_{1,i}}{\partial u_{2}}W_{2,i}^{(0)}+ R_{1,i}W_{1,i+1}^{(0)}, \nonumber
 \end{eqnarray}
 \begin{equation*}
  i=1,2,\ldots,N_{x}-1,\quad W_{1,i}^{(1)}=\mathbf{0},\quad i=0,N_{x}.
 \end{equation*}
 From here, (\ref{eq:4}), (\ref{eq:5}), taking into account that $W_{\alpha,i}^{(0)}\geq \mathbf{0}$, $i=0,1,\ldots,N_{x}$, $\alpha=1,2$, and $R_{1,i}\geq \emph{O}$, we obtain
\begin{equation}\label{eq:17}
  A_{1,i}W_{1,i}^{(1)}+C_{1,i}W_{1,i}^{(1)}\geq L_{1,i}W_{1,i-1}^{(1)},\quad i=1,2,\ldots,N_{x}-1,
\end{equation}
\begin{equation*}
  W_{1,i}^{(1)}=\mathbf{0},\quad i=0,N_{x}.
\end{equation*}
Taking into account that $(A_{1,i}+C_{1,i})^{-1}\geq \emph{O}$, $i=1,2,\ldots,N_{x}-1$,  for $i=1$ in (\ref{eq:17}) and $W_{1,0}^{(1)}=\mathbf{0}$, we conclude that $W_{1,1}^{(1)}\geq \mathbf{0}$. For $i=2$ in (\ref{eq:17}), using $L_{1,2}\geq0$ and $W_{1,1}^{(1)}\geq\mathbf{0}$, we obtain $W_{1,2}^{(1)}\geq \mathbf{0}$. Thus, by induction on $i$, we prove that
\begin{equation*}
W_{1,i}^{(1)}\geq \mathbf{0},\quad i=0,1,\ldots,N_{x}.
\end{equation*}
By following a similar argument, we can prove (\ref{eq:14}) for $\alpha=2$.

We now prove that $\overline{U}_{\alpha,i}^{(1)}$ and $\underline{U}_{\alpha,i}^{(1)}$, $i=0,1,\ldots,N_{x}$, $\alpha=1,2$, are, respectively, upper and lower solutions to (\ref{eq:10}).
From (\ref{eq:10}) for $\alpha=1$ and using the mean-value theorem (\ref{eq:mvth3}), we conclude that
\begin{eqnarray}\label{eq:R1}
\mathcal{K}_{1,i}(\overline{U}_{i}^{(1)})&=&-\left(C_{1,i}-\frac{\partial F_{1,i}(\overline{E}^{(1)}_{1,i},\overline{U}^{(0)}_{2,i})}{\partial u_{1}}\right) \overline{Z}_{1,i}^{(1)}+\frac{\partial F_{1,i}(\overline{U}^{(0)}_{1,i},\overline{E}^{(1)}_{2,i})}{\partial u_{2}}\overline{Z}_{2,i}^{(1)}\nonumber \\
&-&R_{1,i}\overline{Z}_{1,i+1}^{(1)},\quad  i=1,2, \ldots, N_{x}-1,
\end{eqnarray}
where
$$
\overline{U}^{(1)}_{\alpha,i}\leq \overline{E}^{(1)}_{\alpha,i}\leq \overline{U}^{(0)}_{\alpha,i},\quad i=0,1,\ldots,N_{x},\quad \alpha=1,2.
$$
 From (\ref{eq:12}), (\ref{eq:13}) and (\ref{eq:14}), we conclude that $\frac{\partial F_{1,i}}{\partial u_{1}}$ and $\frac{\partial F_{1,i}}{\partial u_{2}}$ satisfy
 (\ref{eq:4}) and (\ref{eq:5}). From (\ref{eq:4}), (\ref{eq:5}), (\ref{eq:12}) and taking into account that $R_{1,i}\geq \emph{O}$, $i=1,2,\ldots,N_{x}-1$, we conclude that
 \begin{equation}\label{eq:k3}
 \mathcal{K}_{1,i}(\overline{U}^{(1)}_{i})\geq \mathbf{0},\quad i=1,2,\ldots,N_{x}-1.
 \end{equation}
  Similarly, we prove that
  \begin{equation}\label{eq:k4}
 \mathcal{K}_{2,i}(\overline{U}^{(1)}_{i})\geq \mathbf{0}, \quad i=1,2,\ldots,N_{x}-1.
 \end{equation}
 From (\ref{eq:3}), (\ref{eq:k3}) and (\ref{eq:k4}), we conclude that $(\overline{U}^{(1)}_{1,i},\overline{U}^{(1)}_{2,i})$, $i=0,1,\ldots,N_{x}$, is an upper solution to (\ref{eq:2}). By following a similar manner, we have
$$
 \mathcal{K}_{1,i}(\underline{U}^{(1)}_{i})\leq \mathbf{0},\quad  \mathcal{K}_{2,i}(\underline{U}^{(1)}_{i})\leq \mathbf{0},\quad i=1,2,\ldots,N_{x}-1,
 $$
 which means  $(\underline{U}^{(1)}_{1,i},\underline{U}^{(1)}_{2,i})$, $i=0,1,\ldots,N_{x}$, is a lower solution to (\ref{eq:2}).

 By induction on $n$, we can prove that $\{\overline{U}^{(n)}_{\alpha,i}\}$, $\{\underline{U}^{(n)}_{\alpha,i}\}$, $i=0,1,\ldots,N_{x}$, $\alpha=1,2$, are, respectively,  monotone decreasing upper and  monotone
 increasing lower sequences of solutions.

Now we prove that the limiting functions of the upper $\{\overline{U}_{\alpha,i}^{(n)}\}$ and lower $\{\underline{U}_{\alpha,i}^{(n)}\}$, $i=0,1,\ldots,N_{x}$, $\alpha=1,2$, sequences are, respectively,  maximal and minimal solutions  of (\ref{eq:2}). From (\ref{eq:m1}), we conclude that $\lim \overline{U}_{\alpha,i}^{(n)}=\overline{U}_{\alpha,i}$, $i=0,1,\ldots,N_{x}$, $\alpha=1,2$,   as $n\rightarrow \infty$ exists and
  \begin{equation}\label{eq:k5}
  \lim_{n\rightarrow \infty}\overline{Z}_{\alpha,i}^{(n)}=\mathbf{0},\quad i=0,1, \ldots, N_{x},\quad \alpha=1,2.
  \end{equation}
Similar to (\ref{eq:R1}), we have
  \begin{eqnarray}\label{eq:18}
\mathcal{K}_{1,i}(\overline{U}_{i}^{(n)})&=&-\left(C_{1,i}-\frac{\partial F_{1,i}(\overline{E}^{(n)}_{1,i},\overline{U}^{(n-1)}_{2,i})}{\partial u_{1}}\right) \overline{Z}_{1,i}^{(n)}-R_{1,i}\overline{Z}_{1,i+1}^{(n)} \\
&+&\frac{\partial F_{1,i}(\overline{U}^{(n-1)}_{1,i},\overline{E}^{(n)}_{2,i})}{\partial u_{2}}\overline{Z}_{2,i}^{(n)},\quad  i=1, \ldots, N_{x}-1, \nonumber
\end{eqnarray}
where
$$
\overline{U}^{(n)}_{\alpha,i}\leq \overline{E}^{(n)}_{\alpha,i}\leq \overline{U}^{(n-1)}_{\alpha,i},\quad i=0,1,\ldots,N_{x},\quad \alpha=1,2.
$$
 By taking the limit of both sides of  (\ref{eq:18}), and using (\ref{eq:12}), it followers that
\begin{equation}\label{eq:k6}
\mathcal{K}_{1,i}(\overline{U}_{i})=\mathbf{0},\quad i=1,2,\ldots,N_{x}-1.
\end{equation}
Similarly, we obtain
\begin{equation}\label{eq:k7}
\mathcal{K}_{2,i}(\overline{U}_{i})=\mathbf{0},\quad i=1,2,\ldots,N_{x}-1.
\end{equation}
From (\ref{eq:k6}) and (\ref{eq:k7}), we conclude that $(\overline{U}_{1,i},\overline{U}_{2,i})$, $i=0,1,\ldots,N_{x}$, is a maximal solution to the nonlinear difference scheme (\ref{eq:2}). In a similar manner, we can prove that
 $$
 \mathcal{K}_{1,i}(\underline{U}_{i})=\mathbf{0},\quad  \mathcal{K}_{2,i}(\underline{U}_{i})=\mathbf{0} ,\quad i=1,2,\ldots,N_{x}-1,
 $$
 which means that $(\underline{U}_{1,i},\underline{U}_{2,i})$, $i=0,\ldots,N_{x}$, is a minimal solution to the nonlinear difference scheme (\ref{eq:2}).

The proof of (\ref{eq:un1}) repeats the proof of (\ref{eq:un}) in Theorem \ref{mc}.
\end{proof}
\section{Construction of initial upper and lower solutions}
We discuss the construction of upper $\widetilde{U}=(\widetilde{U}_{1},\widetilde{U}_{2})$ and lower $\widehat{U}=(\widehat{U}_{1},\widehat{U}_{2})$ solutions which are used as initial iterations in the monotone iterative methods (\ref{eq:6}) and (\ref{eq:10}).
\subsection{Bounded functions}
Assume that the functions $f_{\alpha}(x,y,u)$, $g_{\alpha}(x,y)$, $\alpha=1,2$ in (1) satisfy the following constraints
 \begin{equation}\label{eq:I2}
f_{\alpha}(x,y,\mathbf{0})\leq 0,\quad g_{\alpha}(x,y)\geq 0,\quad \alpha=1,2,
\end{equation}
\begin{equation*}
 f_{\alpha}(x,y,u)\geq -M_{\alpha},\quad  u\geq 0, \quad \alpha=1,2.
 \end{equation*}
 where $M_{\alpha}$, $\alpha=1,2$, are positive constants.
 From here and the definition of a lower solution (\ref{eq:3}), we conclude that the vector function
 \begin{equation}\label{eq:I3}
 \widehat{U}(x_{i},y_{j})=0,\quad (x_{i},y_{j})\in \overline{\Omega}^{h},
 \end{equation}
 is a  lower solution of the nonlinear difference scheme (\ref{eq:2}).
For upper solutions, we introduce the linear problems
\begin{equation}\label{eq:I4}
  \mathcal{L}_{\alpha,ij}\widetilde{U}_{\alpha}(x_{i},y_{j})=M_{\alpha},\quad (x_{i},y_{j})\in \Omega^{h},\quad \alpha=1,2,
\end{equation}
\begin{equation*}
  \widetilde{U}(x_{i},y_{j})=g(x_{i},y_{j}),\quad (x_{i},y_{j})\in \partial \Omega^{h}.
\end{equation*}
\begin{lemma}\label{L:I1}
Suppose that the assumptions in  (\ref{eq:I2}) are satisfied. Then the mesh functions  $\widehat{U}$ and  $\widetilde{U}$  from (\ref{eq:I3}) and (\ref{eq:I4}), are, respectively, ordered lower and upper solutions to (\ref{eq:2}), such that
\begin{equation}\label{eq:I5}
 \widehat{U}(x_{i},y_{j}) \leq \widetilde{U}(x_{i},y_{j}),\quad (x_{i},y_{j})\in \overline{\Omega}^{h}.
\end{equation}
\end{lemma}
\begin{proof}
From (\ref{eq:I2}) and (\ref{eq:I4}),  we have
\begin{equation*}
   \mathcal{L}_{\alpha,ij}\widetilde{U}_{\alpha}(x_{i},y_{j})+f_{\alpha}(x_{i},y_{j},\widetilde{U})\geq 0,\quad (x_{i},y_{j})\in \Omega^{h}, \quad \alpha=1,2,
\end{equation*}
$$
\widetilde{U}(x_{i},y_{j})=g(x_{i},y_{j}), \quad (x_{i},y_{j})\in \partial \Omega^{h}.
$$
Thus, $\widetilde{U}(x_{i},y_{j})$ is an upper solution (\ref{eq:3}).
We now prove that the vector functions $\widetilde{U}(x_{i},y_{j})$ and $\widehat{U}(x_{i},y_{j})$, $(x_{i},y_{j})\in \overline{\Omega}^{h}$,  are ordered upper and lower solutions. Letting $W(x_{i},y_{j})=\widetilde{U}(x_{i},y_{j})-\widehat{U}(x_{i},y_{j})$, $(x_{i},y_{j})\in \overline{\Omega}^{h}$, from (\ref{eq:I3}) and (\ref{eq:I4}), we have
\begin{equation}\label{eq:I6}
  \mathcal{L}_{\alpha,ij}W_{\alpha}(x_{i},y_{j})=M_{\alpha},\quad(x_{i},y_{j})\in \Omega^{h},\quad \alpha=1,2
\end{equation}
\begin{equation*}
 W(x_{i},y_{j})\geq0,\quad (x_{i},y_{j})\in \partial\Omega^{h}.
\end{equation*}
From here, taking into account that $M_{\alpha}$, $\alpha=1,2$, are positive constants  and using the maximum principle in Lemma \ref{lm1}, we conclude that
\begin{equation*}
  W_{\alpha}(x_{i},y_{j})\geq0,\quad (x_{i},y_{j})\in \overline{\Omega}^{h}, \quad\alpha=1,2.
\end{equation*}
Thus, we prove (\ref{eq:I5}).
\end{proof}
\subsection{Constant upper and lower solutions}
Assume that the functions $f_{\alpha}(x,y,u)$, $g_{\alpha}(x,y)$, $\alpha=1,2$, in (\ref{cont1a}) satisfy the following conditions
\begin{equation}\label{eq:I7}
  f_{\alpha}(x,y,\mathbf{0})\leq0,\quad g_{\alpha}(x,y)\geq0,\quad \alpha=1,2.
\end{equation}
The vector function from (\ref{eq:I3}) is a  lower solution to (\ref{eq:2}). We suppose that there exist positive constants $K_{1}$, $K_{2}$ which satisfy the inequalities
\begin{equation}\label{eq:I8}
  f_{\alpha}(x,y,K)\geq0,\quad (x,y)\in \overline{\omega},\quad g_{\alpha}(x,y)\leq K_{\alpha},\quad (x,y)\in \partial \omega,\quad \alpha=1,2,
\end{equation}
where $K=(K_{1},K_{2})$. In the following lemma, we prove that the vector function
\begin{equation}\label{eq:I9}
  \widetilde{U}(x_{i},y_{j})=K,\quad (x_{i},y_{j}) \in \overline{\Omega}^{h},
\end{equation}
is an  upper solution to (\ref{eq:2}).
\begin{lemma}\label{L:I2}
Suppose that the assumptions in  (\ref{eq:I7}) and (\ref{eq:I8}) are satisfied. Then the mesh functions $\widehat{U}=(\widehat{U}_{1},\widehat{U}_{2})$ and $\widetilde{U}=(\widetilde{U}_{1},\widetilde{U}_{2})$  from (\ref{eq:I3}) and (\ref{eq:I9}) are, respectively, ordered lower and upper solutions to (\ref{eq:2}) and satisfy (\ref{eq:I5}).
\end{lemma}
\begin{proof}
From (\ref{eq:I8}) and (\ref{eq:I9}), we have
\begin{equation}\label{eq:I10}
\mathcal{L}_{\alpha,ij}\widetilde{U}_{\alpha}(x_{i},y_{j})+f_{\alpha}(x_{i},y_{j},\widetilde{U})= f_{\alpha}(x_{i},y_{j},K),\quad (x_{i},y_{j})\in \Omega^{h},
\end{equation}
$$
 \alpha=1,2,\quad  g(x_{i},y_{j})\leq K,\quad (x_{i},y_{j})\in \partial\Omega^{h}.
$$
From here and (\ref{eq:I8}), we conclude that $\widetilde{U}(x_{i},y_{j})$, $(x_{i},y_{j})\in \overline{\Omega}^{h}$, is an upper solution (\ref{eq:3}). We now prove that the vector functions $\widetilde{U}(x_{i},y_{j})$ and $\widehat{U}(x_{i},y_{j})$, $(x_{i},y_{j})\in \overline{\Omega}^{h}$,  are ordered upper and lower solutions. Letting $W(x_{i},y_{j})=\widetilde{U}(x_{i},y_{j})-\widehat{U}(x_{i},y_{j})$, $(x_{i},y_{j})\in
\overline{\Omega}^{h}$, from (\ref{eq:3}),  we have
\begin{equation}\label{eq:I11}
  \mathcal{L}_{\alpha,ij}W_{\alpha}(x_{i},y_{j})+ f_{\alpha}(x_{i},y_{j},\widetilde{U})-f_{\alpha}(x_{i},y_{j},\widehat{U})\geq0,\quad(x_{i},y_{j})\in \Omega^{h},
\end{equation}
\begin{equation*}
\alpha=1,2,\quad W(x_{i},y_{j})\geq0,\quad (x_{i},y_{j})\in \partial\Omega^{h}.
\end{equation*}
From (\ref{eq:I3}), (\ref{eq:I7}), (\ref{eq:I8}) and (\ref{eq:I9}), we conclude that
$$
f_{\alpha}(x_{i},y_{j},\widetilde{U})-f_{\alpha}(x_{i},y_{j},\widehat{U})\geq0,\quad (x_{i},y_{j})\in \overline{\Omega}^{h}, \quad \alpha=1,2.
$$
From here, (\ref{eq:I11}) and using the maximum principle in Lemma \ref{lm1}, we conclude that
$$
W_{\alpha}(x_{i},y_{j})\geq0,\quad (x_{i},y_{j})\in \overline{\Omega}^{h},\quad \alpha=1,2.
$$
Thus, we prove (\ref{eq:I5}).
\end{proof}
\subsection{Gas-liquid interaction model}
The following example illustrates the construction of initial upper and lower solutions for a gas-liquid interaction model. Consider the gas-liquid interaction model where a dissolved gas A and a dissolved reactant B
interact in a bounded diffusion medium $\omega$ (see in \cite{P92} for details). The chemical reaction scheme is given by $A+k_{1}B\rightarrow k_{2}P$ and is called the second order reaction, where $k_{1}$ and $k_{2}$ are the rate constants and P is the product. Denote by $z_{1}(x,y)$ and $z_{2}(x,y)$ the concentrations of the dissolved gas $A$ and the reactant $B$, respectively. Then the above reactant
scheme is governed by (\ref{cont1a}) with $\mbox{L}_{\alpha}z_{\alpha}=\varepsilon_{\alpha}\triangle z_{\alpha}$, $f_{\alpha}=\sigma_{\alpha}z_{1}z_{2}$, $\alpha=1,2$, where $\sigma_{1}$ is the rate constant,
$\sigma_{2}=k_{1}\sigma_{1}$. By choosing a suitable positive constant $\rho_{1}>0$ and letting $u_{1}=\rho_{1}-z_{1}\geq0$, $u_{2}=z_{2}$, we have
\begin{equation}\label{ex1}
  f_{1}=-\sigma_{1}(\rho_{1}-u_{1})u_{2},\quad f_{2}=\sigma_{2}(\rho_{1}-u_{1})u_{2},
\end{equation}
and the system (\ref{cont1a}) is reduced to
$$
-\varepsilon_{\alpha}\triangle u_{\alpha}+f_{\alpha}(u_{1},u_{2})=0,\quad (x,y)\in \omega,\quad \alpha=1,2,
$$
$$
u_{1}(x,y)=g_{1}^{*}(x,y)\geq0,\quad u_{2}(x,y)=g_{2}(x,y)\geq0,\quad (x,y)\in\partial \omega,
$$
where $g^{*}_{1}=\rho_{1}-g_{1}\geq0$ and $g_{1}\geq0$ on $\partial \omega$. It is clear from (\ref{ex1}) that $(f_{1},f_{2})$ is quasi-monotone nondecreasing in the rectangle
$$
S_{\rho}=[0,\rho_{1}]\times [0,\rho_{2}]
$$
for any positive constant $\rho_{2}$.

The nonlinear difference scheme (\ref{eq:2}) is reduced to
\begin{equation}\label{ex2}
  \mathcal{L}_{\alpha,ij}U_{\alpha}(x_{i},y_{j})+f_{\alpha}(U)=0,\quad (x_{i},y_{j}) \in \Omega^{h},\quad \alpha=1,2,
  \end{equation}
  \begin{equation*}
 U_{1}(x_{i},y_{j})=g^{*}_{1}(x_{i},y_{j}),\quad  U_{2}(x_{i},y_{j})=g_{2}(x_{i},y_{j}),\quad (x_{i},y_{j})\in \partial \Omega^{h},
\end{equation*}
where $f_{\alpha}$, $\alpha=1,2$, are defined in (\ref{ex1}).

Introduce the following linear problems
\begin{equation}\label{ex3}
  \mathcal{L}_{\alpha,ij}W_{\alpha}(x_{i},y_{j})=0,\quad (x_{i},y_{j})\in\Omega^{h},\quad \alpha=1,2,
\end{equation}
$$
W_{1}(x_{i},y_{j})=g_{1}^{*}(x_{i},y_{j}),\quad W_{2}(x_{i},y_{j})=g_{2}(x_{i},y_{j}),\quad (x_{i},y_{j})\in \partial\Omega^{h}.
$$
We now show that
\begin{equation}\label{ex4}
  (\widetilde{U}_{1},\widetilde{U}_{2})=(\rho_{1},W_{2}),\quad(\widehat{U}_{1},\widehat{U}_{2})=(W_{1},0),
\end{equation}
are, respectively, upper and lower solutions to (\ref{ex2}).
From (\ref{ex1}), (\ref{ex3}) and (\ref{ex4}), we obtain
\begin{equation*}
  \mathcal{L}_{1,ij}\widetilde{U}_{1}+f_{1}( \widetilde{U}_{1},\widetilde{U}_{2})=\mathcal{L}_{1,ij}\ \rho_{1}+f_{1}(\rho_{1},W_{2})=0,\quad (x_{i},y_{j})\in \Omega^{h},
\end{equation*}
\begin{equation*}
  \widetilde{U}_{1}(x_{i},y_{j})=\rho_{1}\geq g^{*}_{1}(x_{i},y_{j}),\quad (x_{i},y_{j})\in \partial\Omega^{h},
\end{equation*}
and
\begin{equation*}
  \mathcal{L}_{2,ij}\widetilde{U}_{2}+f_{2}(\widetilde{U}_{1},\widetilde{U}_{2})=\mathcal{L}_{2,ij}W_{2}+f_{2}(\rho_{1},W_{2})=0,\quad (x_{i},y_{j})\in \Omega^{h},
\end{equation*}
\begin{equation*}
  \widetilde{U}_{2}(x_{i},y_{j})= g_{2}(x_{i},y_{j}),\quad (x_{i},y_{j})\in \partial\Omega^{h}.
\end{equation*}
From here and using (\ref{eq:3}), we conclude that $(\widetilde{U}_{1},\widetilde{U}_{2})$ is an upper solution of (\ref{ex2}). Similarly, we have
\begin{equation*}
  \mathcal{L}_{1,ij}\widehat{U}_{1}+f_{1}(\widehat{U}_{1},\widehat{U}_{2})=\mathcal{L}_{1,ij}W_{1}+f_{1}(W_{1},0)=0,\quad (x_{i},y_{j})\in \Omega^{h},
\end{equation*}
\begin{equation*}
  \widehat{U}_{1}(x_{i},y_{j})=W_{1}(x_{i},y_{j})= g^{*}_{1}(x_{i},y_{j}),\quad (x_{i},y_{j})\in \partial\Omega^{h}.
\end{equation*}
and
\begin{equation*}
  \mathcal{L}_{2,ij}\widehat{U}_{2}+f_{2}(\widehat{U}_{1},\widehat{U}_{2})=\mathcal{L}_{2,ij}\ 0+f_{2}(W_{1},0)=0,\quad (x_{i},y_{j})\in \Omega^{h},
\end{equation*}
\begin{equation*}
  \widehat{U}_{2}(x_{i},y_{j})\leq g_{2}(x_{i},y_{j}),\quad (x_{i},y_{j})\in \partial\Omega^{h}.
\end{equation*}
From here and using (\ref{eq:3}), we conclude that $(\widehat{U}_{1},\widehat{U}_{2})$ is a lower solution of (\ref{ex2}).

Now we prove that
\begin{equation}\label{a}(\widetilde{U}_{1},\widetilde{U}_{2})\geq (\widehat{U}_{1},\widehat{U}_{2}),\quad \mbox{in}\ \overline{\Omega}^{h}.
\end{equation}
From (\ref{ex1}) and (\ref{ex3}), we have $\rho_{1}-W_{1}\geq0$, $\mbox{on} \ \partial\Omega^{h}$. From here, (\ref{ex3}) and (\ref{ex4}), we obtain
\begin{equation*}
  \mathcal{L}_{1,ij}(\widetilde{U}_{1}-\widehat{U}_{1})=\mathcal{L}_{1,ij}(\rho_{1}-W_{1})=0,\quad \mbox{in}\ \Omega^{h},
\end{equation*}
$$
\rho_{1}-W_{1}\geq0, \quad \mbox{on}\ \partial\Omega^{h}.
$$
From here and using Lemma \ref{lm1}, we obtain
 $$
 \rho_{1}\geq W_{1},\quad \mbox{in}\ \overline{\Omega}^{h}.
$$

Now from (\ref{ex3}) and (\ref{ex4}), we have $\widetilde{U}_{2}-\widehat{U}_{2}=W_{2}-0=W_{2}$, and
$$
\mathcal{L}_{2,ij}(\widetilde{U}_{2}-\widehat{U}_{2})=\mathcal{L}_{1,ij}W_{2}(x_{i},y_{j})=0,\quad (x_{i},y_{j})\in \Omega^{h},
$$
$$
W_{2}(x_{i},y_{j})=g_{2}(x_{i},y_{j})\geq0,\quad (x_{i},y_{j})\in \partial \Omega^{h}.
$$
From here and using Lemma \ref{lm1}, we conclude that
$$
W_{2}(x_{i},y_{j})\geq0,\quad (x_{i},y_{j})\in \overline{\Omega}^{h}.
$$
Thus, we prove (\ref{a}).
\par
From (\ref{ex1}), in the sector $\langle \widehat{U}, \widetilde{U}\rangle$, we have
$$
\frac{\partial f_{1}}{\partial u_{1}}(U_{1},U_{2})=\sigma_{1}U_{2}(x_{i},y_{j})\leq \sigma_{1}W_{2}(x_{i},y_{j}),\quad (x_{i},y_{j})\in \overline{\Omega}^{h},
$$
$$
\frac{\partial f_{2}}{\partial u_{2}}(U_{1},U_{2})=\sigma_{2}(\rho_{1}-U_{1}(x_{i},y_{j}))\leq \sigma_{2}\rho_{1},\quad (x_{i},y_{j})\in \overline{\Omega}^{h},
$$
$$
-\frac{\partial f_{1}}{\partial u_{2}}=\sigma_{1}(\rho_{1}-U_{1}(x_{i},y_{j}))\geq \sigma_{1}(\rho_{1}-\rho_{1})\geq0,\quad(x_{i},y_{j})\in \overline{\Omega}^{h},
$$
$$
-\frac{\partial f_{2}}{\partial u_{1}}=\sigma_{2}U_{2}(x_{i},y_{j})\geq0,\quad(x_{i},y_{j})\in \overline{\Omega}^{h},
$$
and the assumptions in (\ref{eq:4}) are satisfied  with
$$
c_{1}(x_{i},y_{j})=\sigma_{1} W_{2}(x_{i},y_{j}),\quad c_{2}(x_{i},y_{j})= \sigma_{2}\rho_{1},\quad (x_{i},y_{j})\in \overline{\Omega}^{h}.
$$
From here and (\ref{ex4}), we conclude that Theorems \ref{mc} and \ref{mc1} hold  for the discrete gas-liquid interaction model (\ref{ex2}).
\section{Comparison of the block monotone Jacobi and block monotone Gauss--Seidel methods}
The following theorem shows that the block monotone Gauss--Seidel  method (\ref{eq:10}) converge not slower than the block monotone Jacobi method (\ref{eq:6}).
\begin{theorem}\label{eq:comp1}
Let $\widetilde{U}=(\widetilde{U}_{1},\widetilde{U}_{2})$ and $\widehat{U}=(\widehat{U}_{1},\widehat{U}_{2})$
 be, respectively, ordered upper and lower solutions (\ref{eq:3}). Assume that the functions $f_{\alpha}(x_{i},y_{j},U)$, $\alpha=1,2$, satisfy (\ref{eq:4}) and (\ref{eq:5}).  Suppose that the sequences $\{(\overline{U}^{(n)}_{\alpha,i})_{J},(\underline{U}^{(n)}_{\alpha,i})_{J}\}$ and $\{(\overline{U}^{(n)}_{\alpha,i})_{GS},(\underline{U}^{(n)}_{\alpha,i})_{GS}\}$, $i=0,1,\ldots,N_{x}$, $\alpha=1,2$, are, respectively, the sequences generated by the block monotone Jacobi method (\ref{eq:6}) and the block monotone Gauss--Seidel method (\ref{eq:10}), where $(\overline{U}^{(0)})_{J}=(\overline{U}^{(0)})_{GS}=\widetilde{U}$ and
  $(\underline{U}^{(0)})_{J}=(\underline{U}^{(0)})_{GS}=\widehat{U}$, then
  \begin{equation}\label{eq:comp2}
(\underline{U}^{(n)}_{\alpha,i})_{J}\leq  (\underline{U}^{(n)}_{\alpha,i})_{GS}\leq (\overline{U}^{(n)}_{\alpha,i})_{GS} \leq (\overline{U}^{(n)}_{\alpha,i})_{J},\quad i=0,1,\ldots,N_{x},\quad \alpha=1,2.
  \end{equation}
\end{theorem}
\begin{proof}
 From (\ref{eq:6}) and (\ref{eq:10}), we have
 \begin{eqnarray*}
   A_{\alpha,i}(U_{\alpha,i}^{(n)})_{J}&+&C_{\alpha,i}(U^{(n)}_{\alpha,i})_{J} =C_{\alpha,i}(U^{(n-1)}_{\alpha,i})_{J}+ L_{\alpha,i}(U_{\alpha,i-1}^{(n-1)})_{J} \\
    &+&R_{\alpha,i}(U_{\alpha,i+1}^{(n-1)})_{J} - F_{\alpha,i}(U_{\alpha,i}^{(n-1)})_{J}+G^{*}_{\alpha,i}, \nonumber
 \end{eqnarray*}
 $$
i=1,2,\ldots,N_{x}-1,\quad \alpha=1,2,\quad (U^{(n-1)}_{i})_{J}=g_{i},\quad i=0,N_{x}.
 $$
  \begin{eqnarray*}
   A_{\alpha,i}(U_{\alpha,i}^{(n)})_{GS}&+&C_{\alpha,i}(U^{(n)}_{\alpha,i})_{GS}= L_{\alpha,i}(U_{\alpha,i-1}^{(n)})_{GS}+ C_{\alpha,i}(U^{(n-1)}_{\alpha,i})_{GS}\\
    &+&R_{\alpha,i}(U_{\alpha,i+1}^{(n-1)})_{GS} - F_{\alpha,i}(U_{\alpha,i}^{(n-1)})_{GS}+G^{*}_{\alpha,i}, \nonumber
 \end{eqnarray*}
  $$
i=1,2,\ldots,N_{x}-1,\quad \alpha=1,2,\quad (U^{(n-1)}_{i})_{GS}=g_{i},\quad i=0,N_{x}.
 $$
From here, letting $\underline{W}^{(n)}_{\alpha,i}=\left(\underline{U}^{(n)}_{\alpha,i}\right)_{GS}-\left(\underline{U}^{(n)}_{\alpha,i}\right)_{J}$, $i=0,1,\ldots,N_{x}$, $\alpha=1,2$, we have
 \begin{eqnarray}\label{eq:comp3}
      A_{\alpha,i}\underline{W}^{(n)}_{\alpha,i}&+&C_{\alpha,i}\underline{W}^{(n)}_{\alpha,i} \nonumber =C_{\alpha,i}\underline{W}^{(n-1)}_{\alpha,i}+L_{\alpha,i}\left((\underline{U}^{(n)}_{\alpha,i-1})_{GS}-(\underline{U}^{(n-1)}_{\alpha,i-1})_{J}\right) \\ 
       &+&R_{\alpha,i}\underline{W}^{(n-1)}_{\alpha,i+1}-F_{\alpha,i}\left((\underline{U}^{(n-1)}_{i})_{GS}\right)+ F_{\alpha,i}\left((\underline{U}^{(n-1)}_{i})_{J}\right),
 \end{eqnarray}
 $$
 i=1,2,\ldots,N_{x}-1,\quad \alpha=1,2,\quad  \underline{W}^{(n-1)}_{i}=0,\quad i=0,N_{x}.
 $$
 By using  Theorem \ref{mc1}, we have $\left(\underline{U}^{(n)}_{\alpha,i}\right)_{GS}\geq \left(\underline{U}^{(n-1)}_{\alpha,i}\right)_{GS}$, $i=0,1,\ldots,N_{x}$, $\alpha=1,2$. From here and  (\ref{eq:comp3}), we conclude that
\begin{eqnarray}\label{eq:comp4}
 A_{\alpha,i}\underline{W}^{(n)}_{\alpha,i}+C_{\alpha,i}\underline{W}^{(n)}_{\alpha,i}&\geq& C_{\alpha,i}\underline{W}^{(n-1)}_{\alpha,i}+L_{\alpha,i}\underline{W}^{(n-1)}_{\alpha,i-1}+R_{\alpha,i}\underline{W}^{(n-1)}_{\alpha,i+1} \nonumber\\
 &-&F_{\alpha,i}\left((\underline{U}^{(n-1)}_{i})_{GS}\right)+F_{\alpha,i}\left((\underline{U}^{(n-1)}_{i})_{J}\right),  
\end{eqnarray}
$$
i=1,2,\ldots,N_{x}-1,\quad \alpha=1,2,\quad W^{(n-1)}_{i}=0,\quad i=0,N_{x}.
$$
Taking into account that $(A_{\alpha,i}+C_{\alpha,i})^{-1}\geq \emph{O}$, $L_{\alpha,i}\geq\emph{O}$, $R_{\alpha,i}\geq\emph{O}$, $i=1,2,\ldots,N_{x}-1$, $\alpha=1,2$, for $n=1$ in (\ref{eq:comp4}), in view of $(\underline{U}^{(0)}_{\alpha,i})_{GS}=(\underline{U}^{(0)}_{\alpha,i})_{J}$ and $\underline{W}^{(0)}_{\alpha,i}=\mathbf{0}$,  we conclude that
\begin{equation*}
  \underline{W}^{(1)}_{\alpha,i}\geq\mathbf{0},\quad i=0,1,\ldots,N_{x},\quad \alpha=1,2.
\end{equation*}
From here, for $n=2$ in (\ref{eq:comp4}), $L_{\alpha,i}\geq\emph{O}$, $R_{\alpha,i}\geq\emph{O}$ and using the monotone property (\ref{mono2}) on $F_{\alpha,i}$, we have
$$
\underline{W}^{(2)}_{\alpha,i}\geq\mathbf{0},\quad i=0,1,\ldots,N_{x},\quad \alpha=1,2.
$$
By induction on $n$, we prove that
$$
\underline{W}^{(n)}_{\alpha,i}\geq\mathbf{0},\quad i=0,1,\ldots,N_{x},\quad \alpha=1,2.
$$
Thus, we prove (\ref{eq:comp2}) for lower solutions. By following the same manner, we can prove (\ref{eq:comp2}) for upper solutions.
\end{proof}


\begin{thebibliography}{99}\addcontentsline{toc}{section}{References} 
\bibitem{Ab79}  Abraham, B., and Plemmons, R., Nonnegative matrices in the mathematical sciences, Academic Press, New York, (1979).
\bibitem{II05} Boglaev, I., A block monotone domain decomposition algorithm for a semilinear convection-diffusion problem, J. Compt. Appl. Math., 173(2005), 259-277.
\bibitem{I06} Boglaev, I., Monotone iterates for solving nonlinear monotone difference schemes, Computing, 78(2006), 17–-30.
\bibitem{I07} Boglaev, I., Monotone iterates for solving coupled systems of nonlinear elliptic equations. I. Farago, P. Vabishchevich, and L. Vulkov, Editors, in Proceedings of the Fourth International Conference on Finite Difference Methods: Theory and Applications, Rousse University, Rousse, (2007), 1--9.
\bibitem{I08} Boglaev, I., Monotone iterates for solving systems of semilinear elliptic equations and applications, ANZIAM J (E), 49(2008), C591--C608.
\bibitem{I10} Boglaev, I., Numerical solutions of coupled systems of nonlinear elliptic equations, Numer. Methods Partial Diff. Eqs., 28(2010), 621--640.
\bibitem{pa65} Parter, S. V., Mildly nonlinear elliptic partial differential equations and their numerical solution, Numer. Math., 7(1965), 113--128.
\bibitem{Pa82} Parter, S. V. and Steuerwalt, M., Block iterative methods for elliptic and parabolic difference equations. SIAM J. Numer. Anal., 19(1982), 1173-1195.
\bibitem{p90} Pao, C. V., Convergence of coupled systems of nonlinear finite difference elliptic equations, Diff. Integ. Eqs., 3(1990), 783-798.
\bibitem{P92} Pao, C. V., Nonlinear  parabolic and elliptic equations, Plenum Press,  New York, (1992).
\bibitem{p93} Pao, C. V., Positive solutions and dynamics of a finite difference reaction–-diffusion system, Numer. Methods Partial Diff. Eqs., 9(1993), 285-311.
\bibitem{p95} Pao, C. V., Block monotone iterative methods for numerical solutions of nonlinear elliptic equations. Numer. Math., 72(1995), 239–-262.
\bibitem{p&X03} Pao, C. V. and Lu, X., Block monotone iterations for numerical solutions of fourth order nonlinear elliptic boundary value problems, SIAM J. Sci Comp., 25(2003), 164-185.
\bibitem{sm2001} Samarskii, A., The theory of difference schemes. Marcel Dekker, New York-Basel, (2001).
\bibitem{s72} Sattinger, D. H., Monotone methods in nonlinear elliptic and parabolic boundary value problems, Indian Univ. Math. J., 21(1972), 979-1000.
\bibitem{Varga2000} Varga, R. S., Matrix iterative analysis, Springer, Berlin, (2000).
\bibitem{w09} Wang, Y., Higher-order monotone iterative methods for finite difference systems of nonlinear reaction-–diffusion-–convection equations, Appl. Numer. Math., 59(2009), 2677-2693.
\bibitem{w11} Wang, Y., Liang, C., and Ravi, P., A block monotone iterative method for numerical solutions of nonlinear elliptic boundary value problems, Numer. Methods Partial Diff. Eqs., 27(2011), 680-701.

\end{thebibliography}
\end{document}